\documentclass[12pt]{amsart}
\usepackage[letterpaper,margin=1.1in]{geometry}
\usepackage{graphicx}
\usepackage{amssymb}
\usepackage[mathscr]{eucal}
\usepackage{amsmath}
\usepackage{amsthm}
\usepackage{bigints}
\usepackage{enumerate}

\newtheorem{theorem}{Theorem}[section]
\newtheorem{lemma}[theorem]{Lemma}
\newtheorem{corollary}[theorem]{Corollary}
\newtheorem{proposition}[theorem]{Proposition}

\theoremstyle{definition}
\newtheorem{definition}[theorem]{Definition}

\theoremstyle{remark}
\newtheorem{remark}[theorem]{Remark}

\newtheorem*{ack}{Acknowledgements}

\newcommand{\be}{\begin{equation}}
\newcommand{\ee}{\end{equation}}
\newcommand{\ben}{\begin{equation*}}
\newcommand{\een}{\end{equation*}}
\newcommand{\bes}{\begin{eqnarray}}
\newcommand{\ees}{\end{eqnarray}}
\newcommand{\besn}{\begin{eqnarray*}}
\newcommand{\eesn}{\end{eqnarray*}}
\newcommand{\txt}{\textrm}
\newcommand{\mres}{\mathbin{\vrule height 1.6ex depth 0pt width
0.13ex\vrule height 0.13ex depth 0pt width 0.7ex}}

\def\Xint#1{\mathchoice
{\XXint\displaystyle\textstyle{#1}}%
{\XXint\textstyle\scriptstyle{#1}}%
{\XXint\scriptstyle\scriptscriptstyle{#1}}%
{\XXint\scriptscriptstyle\scriptscriptstyle{#1}}%
\!\int}
\def\XXint#1#2#3{{\setbox0=\hbox{$#1{#2#3}{\int}$ }
\vcenter{\hbox{$#2#3$ }}\kern-.57\wd0}}

\def\dashint{\Xint-}

\numberwithin{equation}{section}
\numberwithin{figure}{section}

\begin{document}
\title[Poincar\'e-type inequalities and finding good parameterizations]{Poincar\'e-type inequalities and finding good parameterizations}

\author{Jessica Merhej}
\thanks{The author was partially supported by NSF DMS-0856687 and DMS-1361823 grants.}
\address{Department of Mathematics\\ University of Washington\\ Box 354350\\ Seattle, WA 98195}
\email{jem05@uw.edu, j.e.merhej@gmail.com}
\date{May 24th, 2016}

\begin{abstract}
A very important question in geometric measure theory is how geometric features of a set translate into analytic information about it. In 1960, E. R. Reifenberg proved that if a set is well approximated by planes at every point and at every scale, then the set is a bi-H\"older image of a plane. It is known today that Carleson-type conditions on these approximating planes guarantee a bi-Lipschitz parameterization of the set. In this paper, we consider an $n$-Ahlfors regular rectifiable set $M \subset \mathbb{R}^{n+d}$ that satisfies a Poincar\'{e}-type inequality involving the tangential derivative. Then, we show that a Carleson-type condition on the oscillations of the tangent planes of $M$ guarantees that $M$ is contained in a bi-Lipschitz image of an $n$-plane. We also explore the Poincar\'e-type inequality considered here and show that it is in fact equivalent to other Poincar\'e-type inequalities considered on general metric measure spaces. \end{abstract}

\maketitle

\setcounter{tocdepth}{1}	
\tableofcontents

\section{Introduction} \label{Intro}

Finding bi-Lipschitz parameterizations of sets is a central question in areas of geometric measure theory and geometric analysis. A Lipschitz function on a metric space plays the role played by a smooth function on a manifold, and a bi-Lipschitz function plays the role of that of a diffeomorphism. Many concepts in metric spaces, such as metric dimensions and Poincar\'e inequalities, are preserved under bi-Lipschitz mappings. Moreover, a bi-Lipschitz parameterization of a set by Euclidean space leads to its uniform rectifiability. Uniform rectifiability is a quantified version of rectifiability which is well adapted to the study of problems in harmonic analysis on non-smooth sets.\\

The type of parameterizations discussed in this paper first appeared in 1960 when Reifenberg \cite{Re} showed that if a closed set $M \subset \mathbb{R}^{n+d}$ is well approximated by affine $n$-planes at every point and every scale, then $M$ is a bi-H\"older image of $\mathbb{R}^{n}$. Such a set is called a \emph{Reifenberg flat} set. In recent years, there has been renewed interest in this result and its proof. In particular, Reifenberg type parameterizations have been used to get good parameterizations of many spaces such as chord arc surfaces with small constant (see \cite{Se1, Se2}), and limits of manifolds with Ricci curvature bounded from below (see \cite{CC, CN}). Moreover, Reifenberg's theorem has been refined to get better parameterizations of a set: bi-Lipschitz parameterizations (see \cite{DS2}, \cite{To1}, \cite{DT1}, \cite{M1}). In fact, it is well known today, due to the authors of the latter references, that Carleson-type conditions are the correct conditions to study when seeking necessary and sufficient conditions for bi-Lipschitz parameterizations of sets. For example, in \cite{To1}, Toro considers a Carleson condition on the Reifenberg flatness of $M$ that guarantees its bi-Lipschitz parameterization. In \cite{DT1}, David and Toro consider a Carleson condition on the Jones beta numbers $\beta_{\infty}$ and on the (possibly smaller) $\beta_{1}$ numbers that guarantees the same result. In \cite{M1}, the author studies a Carleson-type condition on the oscillation of the unit normals to an $n$-rectifiable set $M$ of co-dimension 1, that guarantees its bi-Lipschitz parameterization. An $n$-rectifiable set $M \subset \mathbb{R}^{n+d}$ is a generalization of a smooth $n$-manifold in $\mathbb{R}^{n+d}$. Rectifiable sets are characterized by having (approximate) tangent planes (see Definition \ref{ats}) at $\mathcal{H}^{n}$-almost every point. Moreover, in the special case when the rectifiable set $M$ has co-dimension 1, then $M$ has an (approximate) unit normal $\nu$ (see Remark \ref{remunitnor}) at $\mathcal{H}^{n}$-almost every point. In fact, in \cite{M1}, the author considers an $n$-Ahlfors regular rectifiable set $M \subset \mathbb{R}^{n+1}$, of co-dimension 1, that satisfies the following Poincar\'e-type inequality for $d=1$ and $\lambda =2$:\\

For all $x \in M$, $r > 0$, and $f$ a Lipschitz function on $\mathbb{R}^{n+d}$, we have 
\be \label{eqp} \dashint_{B_{r}(x)} \left| f(y) - f_{x,r} \right| \, d \mu(y) \leq C_{P} \, r \, \left(\,\dashint_{B_{\lambda r}(x)}|\nabla^{M}f(y)|^{2} \, d \mu(y) \right)^{\frac{1}{2}}, \ee
where $C_{P}$ denotes the Poincar\'{e} constant that appears here, $\lambda \geq 1$ is the dilation constant, $\mu =$  \( \mathcal{H}^{n} \mres M\) is the Hausdorff measure restricted to $M$, $f_{x,r} = \dashint_{B_{r}(x)}f \, d \mu $ is the average of the function $f$ on $B_{r}(x)$, $B_{r}(x)$ is the Euclidean ball in the ambient space $\mathbb{R}^{n+d}$, and $\nabla^{M}f(y)$ denotes the tangential derivative of $f$ (see Definition \ref{deftander}).\\

Then, the author shows that a Carleson-type condition on the oscillation of the unit normal $\nu$ to $M$ guarantee a bi-Lipschitz parameterization of $M$.

\begin{theorem} \label{Th1Merhej} (see \cite{M1}, Theorem 1.5)
Let $M \subset B_{2}(0) \subset \mathbb{R}^{n+1}$ be an $n$-Ahlfors regular rectifiable set containing the origin, and let $\mu =$  \( \mathcal{H}^{n} \mres M\) be the Hausdorff measure restricted to $M$. Assume that $M$ satisfies the Poincar\'{e}-type inequality (\ref{eqp}) with $d=1$ and $\lambda = 2$. There exists $\epsilon_{0} = \epsilon_{0}(n, C_{M}, C_{P})>0$, such that if for some choice of unit normal $\nu$ to $M$, we have

\be \label{103old} \int_{0}^{1} \left(\,\dashint_{B_{r}(x)} |\nu(y) - \nu_{x,r}|^{2} \, d \mu \right) \frac{dr}{r} < \epsilon_{0}^{2}, \quad \txt{for} \,\, x \in M \cap B_{\frac{1}{10^{4}}}(0), \ee 
then  $M \cap B_{\frac{1}{10^{4}}}(0)$ is contained in the image of an affine $n$-plane by a bi-Lipschitz mapping, with bi-Lipschitz constant depending only on $n$, $C_{M}$ and $C_{P}$.
\end{theorem}

In this paper, we generalize Theorem \ref{Th1Merhej} to higher co-dimensions $d$ and arbitrary dilation constants $\lambda \geq 1$. Before stating the theorem, let us introduce some notation. Suppose that $M \subset \mathbb{R}^{n+d}$ is an $n$-Ahlfors regular rectifiable set that satisfies the Poincar\'{e}-type inequality (\ref{eqp}). Fix $x \in M$ and $r >0$. Let $y \in M \cap B_{r}(x)$ such that the approximate tangent plane $T_{y}M$ of $M$ at the point $y$ exists, and denote by $\pi_{T_{y}M}$ the orthogonal projection of $\mathbb{R}^{n+d}$ on $T_{y}M$. Using the standard basis of $\mathbb{R}^{n+d}$, $\{ e_{1} , \ldots , e_{n+d} \}$,  we can view $\pi_{T_{y}M}$ as an $(n+d) \times (n+d)$ matrix whose $j^{th}$ column is the vector is $\pi_{T_{y}M}(e_{j})$. Thus, we denote $\pi_{T_{y}M}$ by the matrix $\big (a_{ij}(y)\big )_{ij}$. Finally, let $A_{x,r} = \big((a_{ij})_{x,r}\big)_{ij}$, be the matrix whose ${ij}^{th}$ entry is the average of the function $a_{ij}$ in the ball $B_{r}(x)$.\\

\begin{theorem} \label{MTT'}
Let $M \subset B_{2}(0) \subset \mathbb{R}^{n+d}$ be an $n$-Ahlfors regular rectifiable set containing the origin, and let $\mu =$  \( \mathcal{H}^{n} \mres M\) be the Hausdorff measure restricted to $M$. Assume that $M$ satisfies the Poincar\'{e}-type inequality (\ref{eqp}). There exist $\epsilon_{0}= \epsilon_{0}(n,d, C_{M},C_{P}) >0$ and $\theta_{0} = \theta_{0}(\lambda) < 1$, such that if

\be \label{103} \int_{0}^{1} \left(\,\dashint_{B_{r}(x)} |\pi_{T_{y}M} - A_{x,r}|^{2} \, d \mu \right) \frac{dr}{r} < \epsilon_{0}^{2} \quad \txt{for} \,\, x \in M \cap B_{1}(0), \ee 
where $|\pi_{T_{y}M} - A_{x,r}|$ denotes the Frobenius norm \footnote{ \hspace{0.1cm} $|\pi_{T_{y}M} - A_{x,r}|^{2} = \txt{trace} \big( (\pi_{T_{y}M} - A_{x,r})^{2} \big) = \displaystyle \sum_{i,j=1}^{n+d} |a_{ij}(y) - (a_{ij})_{x,r}|^{2}$}  of $\pi_{T_{y}M} - A_{x,r}$, then there exists an onto $K$-bi-Lipschitz map $g: \mathbb{R}^{n+d} \rightarrow \mathbb{R}^{n+d}$ where the bi-Lipschitz constant $K = K(n,d,  C_{M},C_{P})$ and an $n$-dimensional plane $\Sigma_{0}$, with the following properties:
\be \label{aa} g(z)= z \quad \txt{when} \,\,\, d(z, \Sigma_{0}) \geq 2, \ee
and
\be  \label{bb} |g(z)-z| \leq C_{0} \epsilon_{0} \quad \txt{for} \,\,\, z \in \mathbb{R}^{n+d}, \ee
where $C_{0}= C_{0}(n,d,C_{M},C_{P})$. Moreover, 
\be \label{cc} g(\Sigma_{0})\,\, \txt{is a} \,\,\, C_{0} \epsilon_{0} \txt{-Reifenberg flat set}, \ee and 
\be \label{contained} M \cap B_{\theta_{0}}(0) \subset g(\Sigma_{0}). \ee
\end{theorem}

Notice that the conclusion of Theorem \ref{MTT'} states that $M$ is (locally) \emph{contained in} a bi-Lipschitz image of an $n$-plane instead of $M$ being exactly a (local) bi-Lipschitz image of an $n$-plane. This is very much expected, since we do not assume that $M$ is Reifenberg flat, and thus we have to deal with the fact that $M$ might have holes.  However, if we assume, in addition to the hypothesis of Theorem \ref{MTT'}, that $M$ is Reifenberg flat, then we do obtain that $M$ is in fact (locally) a bi-Lipschitz image of an $n$-plane. We show this in this paper as a corollary to Theorem \ref{MTT'}.\\

A natural question is whether the hypotheses of Theorem \ref{MTT'}, that is the Ahlfors regularity of $M$, the Poincar\'e inequality (\ref{eqp}), and the Carleson condition (\ref{103}) imply that $M$ is Reifenberg flat. An affirmative answer to this question would directly imply (by the paragraph above) that the conclusion of Theorem \ref{MTT'} should be that $M$ is \emph{exactly} a bi-Lipschitz image of an $n$-plane instead of $M$ being just \emph{contained in} bi-Lipschitz image of an $n$-plane. A negative answer would show that the conclusion of Theorem \ref{MTT'} is the best that we can hope for. It is not surprising that the Poincar\'e inequality (\ref{eqp}) is the correct condition to explore in order to answer this question (which as we discuss below, will turn out negative). In fact, it is already known that (\ref{eqp}) encodes geometric properties of the set $M$.  \\

Let $(M, d_{0}, \mu)$ be a metric measure space, where $M \subset B_{2}(0)$ is an $n$-Ahlfors regular rectifiable set in $\mathbb{R}^{n+d}$,  $\mu =$  \( \mathcal{H}^{n} \mres M\) is the measure that lives on $M$, and $d_{0}$ is the metric on $M$ which is the restriction of the standard Euclidean metric on $\mathbb{R}^{n+d}$. In \cite{M1}, the author proves that the Poincar\'e inequality (\ref{eqp}) implies that $M$ is quasiconvex. More precisely,

\begin{definition} \label{defqc}
A metric space $(X,d)$ is $\kappa_{1}$-quasiconvex if there exists a constant $\kappa_{1} \geq 1$ such that for any two points $x$ and $y$ in $X$, there exists a rectifiable curve $\gamma$ in $X$, joining $x$ and $y$, such that $\txt{length}(\gamma) \leq \kappa_{1} \, d(x,y)$.
\end{definition}

\begin{theorem} (see \cite{M1} Theorem 5.5) \label{mt} \footnote{ Notice that Theorem 5.5 in \cite{M1} is stated and proved in the ambient space $\mathbb{R}^{n+1}$ (so $d=1$) and for $\lambda = 2$. However, the proof of Theorem 5.5 in \cite{M1} is independent from the co-dimension $d$ of $M$. Thus the exact same statement holds here in the higher co-dimension case, and the quasiconvexity constant  $\kappa_{1}$ stays independent of $d$. Moreover, it is very easy to see that Theorem 5.5 in \cite{M1} still holds with arbitrary $\lambda \geq 1$, and in that case, $\kappa_{1}$ would also depend on $\lambda$.
} Let $(M, d_{0}, \mu)$ be as discussed above. Suppose that $M$ satisfies the Poincar\'{e}-type inequality (\ref{eqp}). Then $(M, d_{0}, \mu)$ is $\kappa_{1}$-quasiconvex, with $\kappa_{1}= \kappa_{1}(n, \lambda, C_{M}, C_{P})$.\end{theorem}

There are many Poincar\'e-type inequalities found in literature that imply quasiconvexity  (see for example \cite{Ch}, \cite{DJS}, \cite{K1}, \cite{K2}). To state a couple of the main ones, let $(X, d, \nu)$ be a measure space endowed with a metric $d$ and a positive complete Borel regular measure $\nu$ supported on $X$. Denote by $B^{X}_{r}(x)$ the metric ball in $X$, center $x \in X$ and radius $r>0$. Moreover, assume that $0 < \nu(B_{r}^{X}(x)) < \infty$ for all $x \in X$ and $r>0$.

\begin{definition} (\textbf{\textit{p}-Poincar\'e inequality})\\
Let $p \geq 1$. $(X, d, \nu)$ is said to admit a $p$-Poincar\'e inequality if there exist constants $\kappa\geq1$ and $\lambda\geq1$ such that for any measurable function $u: X \to \mathbb{R}$ and for any upper gradient $\rho$ (see Definition \ref{defuppgrad}) of $u$, the following holds
\be \label{eqp2} \dashint_{B^{X}_{r}(x)} \left| u(y) - u_{B^{X}_{r}(x)} \right| \, d \nu(y) \leq \kappa \, r \, \left(\,\dashint_{B^{X}_{\lambda r}(x)} \rho(y)^{p} \, d \nu(y) \right)^{\frac{1}{p}}, \ee
where $x \in X$, $r>0$, and $u_{B^{X}_{r}(x)}  := \dashint_{B^{X}_{r}(x)} u \, d \nu$.
\end{definition}

\begin{definition} (\textbf{\textit{Lip}-Poincar\'e inequality})\\
Let $p \geq 1$. $(X, d, \nu)$ is said to admit a $Lip$-Poincar\'e inequality if there exist constants $\kappa\geq1$ and $\lambda\geq1$ such that for every Lipschitz function $f$ on $X$, and for every $x \in X$ and $r>0$, we have
\be \label{eqp3}  \dashint_{B^{X}_{r}(x)} \left| f(y) - f_{B^{X}_{r}(x)}\right| \, d \nu(y) \leq \kappa \, r \left(\,\dashint_{B^{X}_{\lambda r}(x)} (Lipf(y))^{p} \, d \nu(y) \right)^{\frac{1}{p}}, \ee
(see Definition \ref{deflip} for the definition of $Lipf$).
\end{definition}

These Poincar\'e inequalities are a-priori different because the right hand side varies according to the notion of ``derivative" used on the metric space. However, Keith has shown (see \cite{K1}, \cite{K2}) that if $(X, d, \nu)$ is a complete metric measure space with $\nu$ a doubling measure, then $(\ref{eqp2})$ and $(\ref{eqp3})$ are equivalent. It turns out that the Poincar\'e-type inequality (\ref{eqp}) is also related to (\ref{eqp2}) and (\ref{eqp3}).\\

In this paper,  we take $(M, d_{0}, \mu)$ as described above and prove that in this setting, the Poincar\'e-type inequalities (\ref{eqp}) (or a more generalized version of it, see (\ref{eqp1'}) below), (\ref{eqp2}), and (\ref{eqp3}) are equivalent. 

\begin{theorem} \label{epi}
Let $p \geq 1$, and let $(M, d_{0}, \mu)$ be a metric measure space, where $M \subset B_{2}(0)$ is an $n$-Ahlfors regular rectifiable set in $\mathbb{R}^{n+d}$,  $\mu =$  \( \mathcal{H}^{n} \mres M\) is the measure that lives on $M$, and $d_{0}$ is the metric on $M$ which is the restriction of the standard Euclidean metric on $\mathbb{R}^{n+d}$.Then, the following are equivalent:

\begin{enumerate} [(i)]

\item  There exist constants $\kappa\geq1$ and $\lambda\geq1$ such that for any measurable function $u: M \to \mathbb{R}$, for any upper gradient $\rho$ of $u$, and for every $x \in M$ and $r>0$, we have
\be \label{eqp2'} \dashint_{B_{r}(x)} \left| u(y) - u_{x,r} \right| \, d \mu(y) \leq \kappa \, r \, \left(\,\dashint_{B_{\lambda r}(x)} \rho(y)^{p} \, d \mu(y) \right)^{\frac{1}{p}}. \ee

\item There exist constants $\kappa \geq 1$, and $\lambda \geq 1$, such that for every Lipschitz function $f$ on $M$, and for every $x \in M$ and $r>0$, we have
\be \label{eqp3'}  \dashint_{B_{r}(x)} \left| f(y) - f_{x,r}\right| \, d \mu(y) \leq \kappa \, r \left(\,\dashint_{B_{\lambda r}(x)} (Lipf(y))^{p} \, d \mu(y) \right)^{\frac{1}{p}}. \ee

\item There exist constants $\kappa \geq 1$, and $\lambda \geq 1$, such that for every Lipschitz function $f$ on $\mathbb{R}^{n+d}$, and for every $x \in M$ and $r>0$, we have
\be \label{eqp1'} \dashint_{B_{r}(x)} \left| f(y) - f_{x,r}\right| \, d \mu(y) \leq \kappa \, r \left(\,\dashint_{B_{\lambda r}(x)} (|\nabla^{M}f|(y))^{p} \, d \mu(y) \right)^{\frac{1}{p}}. \ee
\end{enumerate}
\end{theorem}

Theorem \ref{epi} is interesting in its own right, as it shows that the Poincar\'e inequality (\ref{eqp}) \big (or more generally, (\ref{eqp1'}) \big) is equivalent to the other usual Poincar\'e-type inequalities on metric spaces that imply quasiconvexity. Moreover, Theorem \ref{epi} opens the door to many examples of spaces satisfying the Poincar\'e inequality (\ref{eqp1'}) as there are many examples in literature of spaces satisfying the $p$-Poincar\'e and $Lip$-Poincar\'e inequalities (see for example \cite{BS}, \cite{HK}, \cite{BB}, \cite{La}). This allows us to get an example of a set that is not Reifenberg flat, and yet satisfies all the hypotheses of Theorem \ref{MTT'}.

\begin{theorem} \label{construct}
There exists a non-Reifenberg flat, $n$-Ahlfors regular, rectifiable set $M \subset B_{2}(0) \subset \mathbb{R}^{n+d}$ that satisfies all the hypotheses of Theorem \ref{MTT'}.
\end{theorem}

Theorem \ref{construct} shows that the hypotheses of Theorem \ref{MTT'} on the set $M$ are not strong enough to guarantee its Reifenberg flatness, and thus the conclusion of Theorem \ref{MTT'} is optimal.\\

The paper is structured as follows: in Section \ref{Pre}, we introduce some definitions and preliminaries. In Section \ref{secMT}, we prove Theorem \ref{MTT'}. Moreover, we prove that Theorem \ref{Th1Merhej} follows as a corollary from Theorem \ref{MTT'}. Section \ref{PIQ} is dedicated to proving that the Poincar\'e inequality (\ref{eqp1'}) is equivalent to the $p$-Poincar\'e and the $Lip$-Poincar\'e inequalities. Finally, in the last section, we prove Theorem \ref{construct} by constructing a concrete example of a set that is not Reifenberg flat, yet satisfies the hypotheses of Theorem \ref{MTT'}.

\section{Preliminaries} \label{Pre}
Throughout this paper, our ambient space is $\mathbb{R}^{n+d}$. $B_{r}(x)$ denotes the open ball center $x$ and radius $r$ in $\mathbb{R}^{n+d}$, while $\bar{B}_{r}(x)$ denotes the closed ball center $x$ and radius $r$ in $\mathbb{R}^{n+d}$. $d(.,.)$ denotes the distance function from a point to a set. $\mathcal{H}^{n}$ is the $n$-Hausdorff measure. Finally, constants may vary from line to line, and the parameters they depend on will always be specified in a bracket. For example, $C(n,d)$ will be a constant that depends on $n$ and $d$ that may vary from line to line. \\

We begin by the definitions needed starting section \ref{secMT} and onwards.

\begin{definition}
Let $M \subset \mathbb{R}^{N_{1}}$. A function $ f: M \rightarrow \mathbb{R}^{N_{2}}$ is called \emph{Lipschitz} if there exists a constant $K>0$, such that for all $x, \, y \in M$ we have
\begin{equation} \label{lip1}
|f(x) - f(y)| \leq K \, |x-y|.
\end{equation}
The smallest such constant is called the \emph{Lipschitz constant} and is denoted by $L_{f}$.
\end{definition}

\begin{definition}
A function $ f: \mathbb{R}^{N_{1}} \rightarrow \mathbb{R}^{N_{2}}$ is called $K$-\emph{bi-Lipschitz} if there exists a constant $K>0$, such that for all $x, \, y \in \mathbb{R}^{N_{1}}$ we have
\begin{center}
$K^{-1} |x-y| \leq |f(x) - f(y)| \leq K \, |x-y|.$ 
\end{center}
\end{definition}

Let's introduce the class of \emph{n-rectifiable} sets, and the definition of approximate tangent planes.

\begin{definition} \label{rect}
Let $M \subset \mathbb{R}^{n+d}$ be an $\mathcal{H}^{n}$-measurable set. $M$ is said to be countably \emph{n-rectifiable} if 
 \begin{center}$ M \subset M_{o} \cup \left(\displaystyle \bigcup_{i=1}^{\infty}f_{i}(A_{i})\right)$, \end{center}
where $ \mathcal{H}^{n}(M_{o}) = 0$, and $ f_{i} : A_{i} \rightarrow \mathbb{R}^{n+d}$ is Lipschitz, and $A_{i} \subset \mathbb{R}^{n}$, for  $i = 1, 2, \ldots$ 
\end{definition}

\begin{definition} \label{ats}
If $M$ is an $\mathcal{H}^{n}$-measurable subset of $\mathbb{R}^{n+d}$. We say that the $n$-dimensional subspace $P(x)$ is the \emph{approximate tangent space of $M$ at $x$}, if
\begin{equation}   \lim_{h \to 0}  h^{-n} \int_{M} {f \left (h^{-1}(y-x)\right) } \, d \mathcal{H}^{n}(y) = \int_{P(x)} f(y) \, d\mathcal{H}^{n}(y) \quad \forall f \in C^1_c(\mathbb{R}^{n+d}, \mathbb{R}).
\end{equation} 
\end{definition}

\begin{remark} \label{remunitnor}
Notice that if it exists, $P(x)$ is unique. From now on, we shall denote the tangent space of $M$ at $x$ by $T_{x}M$. Moreover, in the special case when $M$ has co-dimension 1, then one can define the unit normal $\nu$ to $M$ at the point $x \in M$ to be the unit normal to $T_{x}M$. Thus, the unit normal $\nu$ exists at every point $x \in M$ that admits a tangent plane, and of course, there are two choices for the direction of the unit normal.

\end{remark}

It is well known (see \cite{Si}; Theorem 11.6) that $n$-rectifiable sets have tangent planes at $\mathcal{H}^{n}$ almost every point in the set.

\begin{definition} \label{deftander}
Let $f$ be a real valued Lipschitz function on $\mathbb{R}^{n+d}$. The tangential derivative of $f$ at the point $y \in M$ id denoted by $\nabla^{M}f(y)$ and defined as follows:

\begin{equation} \label{td'} \nabla^{M}f(y) = \nabla (f|_{L}) (y)\end{equation}
where $L := y + T_{y}M$, $f|_{L}$ is the restriction of $f$ on the affine subspace $L$, and $\nabla(f|_{L})$ is the usual gradient of $f|_{L}$.\\

In the special case when $f$ is a smooth function on $\mathbb{R}^{n+d}$, we have
\begin{equation} \label{td} \nabla^{M}f(y) = \pi_{T_{y}M} (\nabla f (y)),\end{equation}
where $\pi_{T_{y}M}$ is the orthogonal projection of $\mathbb{R}^{n+1}$ on $T_{y}M$, and $\nabla f$ is the usual gradient of $f$.
\end{definition}

Note that $\nabla^{M}f(y)$ exists at $\mathcal{H}^{n}$- almost every point in $M$.\\

We also need to define the notion of \emph{Reifenberg flatness}: 
\begin{definition} Let $M$ be an $n$-dimensional subset of $\mathbb{R}^{n+d}$. We say that $M$ is $\epsilon$-Reifenberg flat for some $\epsilon >0$, if for every $x \in M$   and $0 < r \leq \frac{1}{10^{4}}$, we can find an $n$-dimensional affine subspace $P(x,r)$ of $\mathbb{R}^{n+d}$ that contains $x$ such that 

\begin{equation*}
d(y, P(x,r)) \leq \epsilon r \quad \txt{for} \,\,y \in M \cap B_{r}(x),
\end{equation*} 
and 
\begin{equation*} 
d(y, M) \leq \epsilon r \quad \txt{for} \,\, y \in P(x,r) \cap B_{r}(x).
\end{equation*} 
 \end{definition}

\begin{remark}
Notice that the above definition is only interesting if $\epsilon$ is small, since any set is 1-Reifenberg flat. \end{remark}

In the proof of our Theorem \ref{MTT'}, we need to measure the distance between two $n$-dimensional planes. We do so in terms of normalized local Hausdorff distance:

\begin{definition}
Let $x$ be a point in $\mathbb{R}^{n+d}$ and let $r >0$. Consider two closed sets $E,\,F \subset \mathbb{R}^{n+d}$ such that both sets meet the ball $B_{r}(x)$. Then,
\begin{equation*} d_{x,r}(E,F) = \frac{1}{r} \, \txt{Max} \left\{ \sup_{y \in E \cap B_{r}(x)} \txt{dist}(y,F) \,\,; \sup_{y \in F \cap B_{r}(x)} \txt{dist}(y,E) \right\} \end{equation*}
is called the normalized Hausdorff distance between $E$ and $F$ in $B_{r}(x)$.
\end{definition}

Let us recall the definition of an $n$-Ahlfors regular measure and an $n$-Ahlfors regular set:

\begin{definition}
Let $M \subset \mathbb{R}^{n+d}$ be a closed, $\mathcal{H}^{n}$ measurable set, and let $\mu=$  \( \mathcal{H}^{n} \mres M\) be the $n$-Hausdorff measure restricted to $M$. We say that $\mu$ is $n$-Ahlfors regular if there exists a constant $C_{M} \geq 1$, such that for every $x \in M$ and $ 0< r \leq 1$, we have 
\begin{equation} \label{alfh}  C_{M}^{-1} \, r^{n} \leq \mu(B_{r}(x)) \leq C_{M} \, r^{n}.\end{equation}

In such a case, the set $M$ is called an $n$-Ahlfors regular set, and $C_{M}$ is referred to as the Ahlfors regularity constant.
\end{definition}

Let us now move to definitions and notations needed in sections \ref{PIQ} and \ref{MCSHH}. In these sections, $(X,d)$ denotes a space $X$ endowed with a metric $d$. $B^{X}_{r}(x)$ denotes the open metric ball of center $x \in X$ and radius $r>0$. Moreover, $(X, d, \nu)$ denotes a measure space endowed with a metric $d$ and a positive complete Borel regular measure $\nu$ supported on $X$ such that $0 < \nu(B_{r}^{X}(x)) < \infty$ for all $x \in X$ and $r>0$.

\begin{definition}
Let $(X,d, \nu)$ be a metric measure space. We say that $\nu$ is a doubling measure if there is a constant $\kappa_{0} >0$ such that 
\ben \nu\left(B^{X}_{2r}(x)\right)\leq \kappa_{0} \, \nu\left(B^{X}_{r}(x)\right), \een where $x \in X$, $r>0$.
\end{definition}

In sections \ref{PIQ} and \ref{MCSHH}, a curve $\gamma$ in a metric space $(X,d)$ is a continuous non-constant map from a compact interval $I \subset \mathbb{R}$ into $X$. $\gamma$ is said to be rectifiable if it has finite length, where the latter is denoted by $l(\gamma)$. Thus, any rectifiable curve can be parametrized by arc length, and we will always assume that it is.\\

Let us now define the notions of upper gradients, $p$-weak upper gradients, and the Local Lipschitz constant function. 

\begin{definition} \label{defuppgrad}
A non-negative Borel function $\rho: X \rightarrow [0, \infty]$ is said to be an \emph{upper gradient} of a function $u: X \rightarrow \mathbb{R}$ if 
\ben  |u(\gamma(0)) - u(\gamma(l_{\gamma}))| \leq \int_{\gamma} \rho \, ds,  \een
for any rectifiable curve $\gamma : [0, l_{\gamma}] \to X$.
\end{definition}

\begin{definition}
Let $p \geq 1$ and let $\Gamma$ be a family of rectifiable curves on $X$. We define the \emph{$p$-modulus} of $\Gamma$ by 
\ben \txt{Mod}_{p}(\Gamma) = \inf \int_{X} g^{p} \, d \nu \een
where the infimum is taken over all nonnegative Borel functions $g$ such that $\int_{\gamma} g \, ds \geq 1$ for all $\gamma \in \Gamma$.
\end{definition}

\begin{definition}
A non-negative measurable function $\rho: X \rightarrow [0, \infty]$ is said to be a \emph{p-weak upper gradient} of a function $u: X \rightarrow \mathbb{R}$ if 
\ben  |u(\gamma(0)) - u(\gamma(l_{\gamma}))| \leq \int_{\gamma} \rho \, ds,  \een
for $p$-a.e. rectifiable curve $\gamma : [0, l_{\gamma}] \to X$ (that is, with the exception of a curve family of zero $p$-modulus).

\end{definition}

\begin{definition} \label{deflip} Let $f$ be a Lipschitz function on a metric measure space $(X,d,\nu)$. The local Lipschitz constant function of $f$ is defined as follows
\be \label{lip} Lipf(x) = \lim_{r \to 0} \sup_{y \in B^{X}_{r}(x), \,y \neq x} \frac{|f(y) - f(x)|}{d(y,x)}, \,\,\,\,\, x \in X,\ee
where $B_{r}^{X}(x)$ denotes the metric ball in $X$, center $x$, and radius $r$.
\end{definition}

\begin{remark} Let us note here that for any Lipschitz function $f$, $L_{f}$ denotes the usual Lipschitz constant (see sentence below (\ref{lip1})), whereas $Lipf(.)$ stands for the local Lipschitz constant function defined above. \end{remark}

\section{A bi-Lipschitz parameterization of $M$} \label{secMT}

The main goal in this section is to prove Theorem \ref{MTT'}. We begin with three linear Algebra lemmas needed to prove the theorem, as they can be stated and proved independently. 

\begin{lemma} \label{l4}
In the next lemma, let $V$ be an $n$-dimensional subspace of $\mathbb{R}^{n+d}$. Denote by $\pi_{V}$ the orthogonal projection on $V$. Then, there exists a $\delta_{0}= \delta_{0}(n,d) > 0$, such that for any $\delta \leq \delta_{0}$, and for any linear operator $L$ on $\mathbb{R}^{n+d}$ such that 
\be \label{l41} || \pi_{V} - L || \leq \delta , \ee
where $|| . ||$ denotes the induced operator norm, $L$ has exactly $n$ eigenvalues  $\lambda_{1}, \ldots , \lambda_{n}$ such that
\be \label{l42'} |\lambda_{j}| \geq 1 - (n+d) \, \delta \geq \frac{3}{4}, \quad \forall \, j \in \{ 1 , \ldots, n\}, \ee
and exactly  $d$ eigenvalues $\lambda_{n+1}, \ldots , \lambda_{n+d}$, such that 
\be \label{l42} |\lambda_{j}| \leq (n+d) \, \delta \leq \frac{1}{4}, \quad \forall \, j \in \{ n+1 , \ldots, n+d\}.  \ee
\end{lemma}

\begin{proof}
Since $\pi_{V}$ is an orthogonal projection, then there exists an orthonormal basis $\{ w_{1} , \ldots , w_{n+d}  \} $ of $\mathbb{R}^{n+d}$ such that the matrix representation of $\pi_{V}$ in this basis is 
\[
\pi_{V} =
  \begin{pmatrix}
    Id_{n} & 0 \\
    0 & 0
 \end{pmatrix}
\]
where $Id_{n}$ denotes the $n \times n$ identity matrix. \\
\\
Let $\delta < \delta_{0}$ (with $\delta_{0}$ to be determined later), and suppose $L$ is as in the statement of the lemma. Let $L = (l_{ij})_{ij}$ be the matrix representation of $L$ in the basis $\{ w_{1} , \ldots , w_{n+d}  \} $. Then, by (\ref{l41}), we have 
\ben |\pi_{V} w_{j} - L w_{j} |^{2} \leq \delta^{2}, \quad \forall \, j \in \{ 1 \ldots n+d \}, \een
that is,
\be \label{l43} |1 - l_{jj}|^{2} + \sum_{i\neq j} |l_{ij}|^{2} \leq \delta^{2}, \quad \forall \, j \in \{ 1 \ldots n \}, \ee
and 
\be \label{l44} \sum_{i=1} ^{n+d} |l_{ij}|^{2} \leq \delta^{2}, \quad \forall \, j \in \{ n+1 \ldots n+d \}. \ee
Now, for each $ j \in  \{ 1 \ldots n+d\}$, consider the closed disk $D_{j}$ in the complex plane, of center $(l_{jj}, 0)$ and radius $R_{j} = \displaystyle \sum_{i \neq j} |l_{ij}|$. Notice that by (\ref{l43}), (\ref{l44}), and the fact that $\delta < \delta_{0}$, we have 
\be \label{l45'} |1 - l_{jj} | \leq \delta \leq \delta_{0}, \quad \forall \, j \in \{ 1 \ldots n \}, \ee 
\be \label{l45} |l_{jj}| \leq \delta \leq \delta_{0}, \quad \forall \, j \in \{ n+1 \ldots n+d \}, \ee
and \be \label{l46} R_{j} \leq (n+d-1) \delta \leq (n+d-1) \, \delta_{0}, \quad \forall \, j \in \{1 \ldots n+d \}. \ee
Choosing $\delta_{0}$ such that $(n+d-1) \delta_{0} \leq \displaystyle \frac{1}{8}$ , we can guarantee that $\displaystyle \bigcup_{j=1}^{n}D_{j}$ is disjoint from $\displaystyle \bigcup_{j=n+1}^{n+d}D_{j}$. Thus, by the Gershgorin circle theorem (see \cite{LeV}, p.277-278), $\displaystyle \bigcup_{j=1}^{n}D_{j}$ contains exactly $n$ eigenvalues of $L$, and $\displaystyle \bigcup_{j=n+1}^{n+d}D_{j}$ contains exactly $d$ eigenvalues of $L$. The lemma follows from (\ref{l45'}), (\ref{l45}) and (\ref{l46})
\end{proof}

\textbf{Notation:}\\
Let $V$ be an affine subspace of $\mathbb{R}^{n+d}$ of dimension $k$, $ k \in \{ 0, \dots, n-1\}$. 
Denote by $N_{\delta}(V)$, the $\delta$-neighborhood of $V$, that is, 
\ben N_{\delta}(V) = \left\{ x \in \mathbb{R}^{n+1} \,\,\txt{such that} \,\, d(x,V) < \delta \right\}. \een

\begin{lemma} \label{l1} (see \cite{M1}, Lemma 3.1) \footnote{Notice that Lemma 3.1 in \cite{M1} is stated and proved in the ambient space $\mathbb{R}^{n+1}$, whereas Lemma \ref{l1} here has $\mathbb{R}^{n+d}$ as the ambient space. However, one can very easily adapt the same proof of Lemma 3.1 in \cite{M1} to this higher co-dimension case here, while noticing that $c_{0}$ in the latter case should also depend on the co-dimension $d$.} Let $M$ be an $n$-Ahlfors regular subset of $\mathbb{R}^{n+d}$, and let $\mu =$  \( \mathcal{H}^{n} \mres M\) be the Hausdorff measure restricted to $M$. There exists a constant $c_{0} = c_{0}(n,d, C_{M}) \leq \displaystyle \frac{1}{2}$ such that the following is true: Fix $x_{0} \in M$, $r_{0} < 1$ and let $r = c_{0} \, r_{0}$. Then, for every $V$, an affine subspace of $\mathbb{R}^{n+d}$ of dimension $0 \leq k \leq n-1$,
there exists $x \in M \cap B_{r_{0}}(x_{0})$ such that $x \notin  N_{11 r}(V)$ and $B_{r}(x) \subset B_{2 r_{0}}(x_{0})$. 
\end{lemma}

\begin{lemma} \label{l3} (see \cite{M1} Lemma 3.3) \footnote{Notice that Lemma 3.3 in \cite{M1} is stated and proved in the ambient space $\mathbb{R}^{n+1}$, whereas Lemma \ref{l3} here has $\mathbb{R}^{n+d}$ as the ambient space. However, the proof of Lemma 3.3 in \cite{M1} is in fact independent from the co-dimension $d$ of $M$. Thus the exact same proof holds here, and the constant $K_{1}$ stays independent of $d$.} Fix $R>0$, and let $\{u_{1}, \ldots u_{n} \}$ be $n$ vectors in $\mathbb{R}^{n+d}$. Suppose there exists a constant $K_{0} >0$ such that 

\be \label{44'} |u_{j}| \leq K_{0} \, R \quad \forall j \in \{1,\ldots, n\}. \ee
Moreover, suppose there exists a constant $0 < k_{0} < K_{0}$, such that 
\be \label{44}|u_{1}|\geq k_{0} \, R,\ee
and 
\be \label{45} u_{j} \notin N_{k_{0}R}\big(span\{u_{1}, \ldots u_{j-1}\} \big)  \quad \forall j \in \{2,\ldots, n\}. \ee

Then, for every vector $v \in V:= span\{u_{1}, \ldots u_{n}\} $, $v$ can be written uniquely as 
\be v = \sum_{j=1}^{n} \beta_{j}u_{j},\ee
where \be \label{46} |\beta_{j}| \,\leq K_{1}\frac{1}{R} \, |v|, \quad \forall j \in \{1,\ldots, n\} \ee
with $K_{1}$ being a constant depending only on $n$, $k_{0}$, and $K_{0}$.
\end{lemma}

Throughout the rest of the paper, $M$ denotes an $n$-Ahlfors regular rectifiable subset of $\mathbb{R}^{n+d}$ and $\mu =$  \( \mathcal{H}^{n} \mres M\) denotes the Hausdorff measure restricted to $M$. The average of a function $f$ on the ball $B_{r}(x)$ is denoted by 
\begin{equation} \label{average} f_{x,r}= \dashint_{B_{r}(x)} f  \, d \mu(y) =  \displaystyle \frac{1}{\mu(M \cap B_{r}(x))}   \int_{B_{r}(x)} f \, d \mu(y). \end{equation}

We recall the statement of Theorem \ref{MTT'}: if $M$ satisfies the Poincar\'e-type condition (\ref{eqp}), and if the Carleson-type condition (\ref{103}) on the oscillation of the tangent planes to $M$ is satisfied, and if then $M$ is contained in a bi-Lipschitz image of an $n$-dimensional plane.\\

To prove this theorem, we follow steps similar to those used in \cite{M1} to prove the co-dimension 1 case (see Theorem 1.5 in \cite{M1}) which is stated as Theorem \ref{Th1Merhej} in this paper. First, we define what we call the $\alpha$-numbers

\be \label{a'} \alpha(x,r) := \left(\,\dashint_{B_{r}(x)} |\pi_{T_{y}M} - A_{x,r}|^{2} \, d \mu \right)^{\frac{1}{2}}, \ee
where $x \in M$, and $0 < r \leq \displaystyle \frac{1}{10}$, $\pi_{T_{y}M}$ has $\big ( a_{ij}(y) \big)_{ij}$ as its matrix representation in the standard basis of $\mathbb{R}^{n+d}$, and $A_{x,r}= \big( (a_{ij})_{x,r} \big)_{ij}$  is the matrix whose ${ij}^{th}$ entry is the average of the function $a_{ij}$ in the ball $B_{r}(x)$.\\

These numbers are the key ingredient to proving our theorem. In Lemma \ref{l5}, we show that  the Carleson condition (\ref{103}) implies that these numbers are small at every point $x \in M$ and every scale $0 < r < \frac{1}{10}$. Moreover, for every point $x \in M$, and series $\displaystyle \sum_{i=1}^{\infty} \alpha^2(x,  10^{-j})$ is finite. Then, in Theorem \ref{t1}, we use the Poincar\'e-type inequality to get an $n$-plane $P_{x,r}$ at every point $x \in M$ and every scale $0<r \leq\frac{1}{10 \lambda}$  such that the distance (in integral form) from $M \cap B_{r}(x)$ to $P_{x,r}$ is bounded by $\alpha(x,\lambda r)$. This means, by Lemma \ref{l5}, that those distances are small, and for a fixed point $x$, when we add these distances at the scales $ 10^{-j}$ for $j \in \mathbb{N}$, this series is finite \footnote{ A note for the interested reader: Theorem \ref{t1} implies that the series $\displaystyle \sum_{i=1}^{\infty} \beta_{1}^2(x,  10^{-j})$ is finite. See \cite{M1} on how this relates to the $\beta_{1}$-numbers, and the theorems found in \cite{DT1} that involve a Carleson condition on the $\beta_{1}$-numbers that guarantees a bi-Lipschitz parameterization of the set.}. Theorem \ref{t1} is the key point that allows us to use the bi-Lipschitz parameterization that G. David and T. Toro construct in \cite{DT1}. In fact, what they do is construct approximating $n$-planes, and prove that at any two points that are close together, the two planes associated to these points at the same scale, or at two consecutive scales are close in the Hausdorff distance sense. From there, they construct a bi-H\"older  parameterization for $M$. Then, they show that the sum of these distances at scales $ 10^{-j}$ for $j \in \mathbb{N}$ is finite (uniformly for every $x \in M$). This is what is needed for their parameterization to be bi-Lipschitz (see Theorem \ref{t2} below and the definition before it). Thus, the rest of the proof is devoted to using Theorem \ref{t1} in order to prove the compatibility conditions between the approximating planes mentioned above.\\

Note that, in the process of proving Theorem \ref{MTT'}, we find several parts of the proof very similar to the proof of the co-dimension 1 case found in \cite{M1} (see Theorem 1.5 in \cite{M1} or Theorem \ref{Th1Merhej} in this paper). In fact, most of the differences in the proof happen in Lemma \ref{l5} and Theorem \ref{t1}, with the most important difference being in the latter. The rest of the proof follows closely to the proof of co-dimension 1 case. Thus, in this paper we do as follows: first, we prove Lemma \ref{l5} and Theorem \ref{t1} and include all the details. Then, for the rest of the proof (that is introducing the David and Toro bi-Lipschitz construction, and proving the compatibility conditions between the approximating planes that allow us to use this construction), we only give an outline of the main ideas, and leave the smaller details and tedious calculations out. However, in each place where the details are omitted, we refer the reader to the parts of the proof of Theorem 1.5 in \cite{M1} where they can be found. That being said, this part of the proof of Theorem \ref{MTT'} still has enough details so that the reader understands all the steps needed to get the bi-Lipschitz parameterization of $M$, and the intuition behind them. Moreover, the way the proof is presented here includes all the information that we need from the construction of the bi-Lipschitz parameterization of $M$ to prove the corollaries that follow from Theorem \ref{MTT'}.\\

Let us begin with Lemma \ref{l5} that decodes the Carleson condition (\ref{103}).

\begin{lemma} \label{l5}
Let $M \subset B_{2}(0)$ be an $n$-Ahlfors regular rectifiable set containing the origin, and let $\mu =$  \( \mathcal{H}^{n} \mres M\) be the Hausdorff measure restricted to $M$. Let $\epsilon > 0$, and suppose that \be \label{0} \int_{0}^{1} \left(\,\dashint_{B_{r}(x)}|\pi_{T_{y}M} - A_{x,r}|^{2}  \, d \mu \right) \frac{dr}{r} < \epsilon^{2}, \quad \forall \, x \in M. \ee
Then, for every $x \in M$, we have 
\be \label{a} \sum_{k=1}^{\infty} \alpha^{2}(x, 10^{-k}) \leq C \, \epsilon^{2} ,\ee where the $\alpha$-numbers are as defined in (\ref{a'}) and $C = C(n, C_{M})$. Moreover, for every $x \in M$ and $0 < r \leq \displaystyle \frac{1}{10}$, we have
\be \label{a''} \alpha(x,r) \leq C \, \epsilon ,\ee
 where $C = C(n, C_{M})$.\end{lemma}

\begin{proof}
Let $\epsilon > 0$ and suppose that (\ref{0}) holds. By the definition of the Frobenius norm, (\ref{0}) becomes 
\be \label{l11} \sum_{i,j=1}^{n+d}  \int_{0}^{1} \left(\,\dashint_{B_{r}(x)}| a_{ij}(y) - (a_{ij})_{x,r}|^{2}  \, d \mu \right) \frac{dr}{r} < \epsilon^{2}, \quad \forall \, x \in M,\ee
where $\pi_{T_{y}M} = \big(a_{ij}(y)\big)_{ij}$ and $A_{x,r} = \big((a_{ij})_{x,r}\big)_{ij}$.\\
\\
 Fix $x \in M$, and fix $i, \, j \in \{ 1, \ldots n+d\}$. For all $a \in \mathbb{R}$, and for all $0 < r_{0} \leq 1$, we have
\be \label{10} \dashint_{B_{r_{0}}(x)} |a_{ij}(y) - (a_{ij})_{x,r_{0}}|^{2} \, d \mu \leq  \,\dashint_{B_{r_{0}}(x)} |a_{ij}(y) - a|^{2} \, d \mu,\ee
since the average $(a_{ij})_{x,r_{0}}$ of $a_{ij}$ in the ball $B_{r_{0}}(x)$ minimizes the integrand on the right hand side of (\ref{10}).\\
To prove (\ref{a}), we note that
 \be \label{17} \sum_{k=1}^{\infty}  \dashint_{B_{ 10^{-k}}(x)} |a_{ij}(y) -  (a_{ij})_{x, 10^{-k}}|^{2} \, d \mu  \leq C(n,C_{M}) \, \sum_{k=0}^{\infty} \int_{ 10^{-k-1}}^{ 10^{-k}} \dashint_{B_{r}(x)} |a_{ij}(y) - (a_{ij})_{x,r}|^{2} \, d \mu \, \frac{dr}{r}.\ee

This is a straightforward computation that uses (\ref{10}) and the Ahlfors regularity of $\mu$, and is found in details in \cite{M1} (see \cite{M1}, Lemma 4.1 proof of inequality (4.6)). Moreover, it is trivial to check that 
\be \label{baa} \sum_{k=0}^{\infty} \int_{ 10^{-k-1}}^{ 10^{-k}} \dashint_{B_{r}(x)} |a_{ij}(y) - (a_{ij})_{x,r}|^{2} \, d \mu \, \frac{dr}{r} = \int_{0}^{1} \left(\,\dashint_{B_{r}(x)} |a_{ij}(y) - (a_{ij})_{x,r}|^{2} \, d \mu \right) \frac{dr}{r}.\ee

Thus, plugging (\ref{baa}) in (\ref{17}), we get 
\be \label{l120} \sum_{k=1}^{\infty} \dashint_{B_{ 10^{-k}}(x)} |a_{ij}(y) - (a_{ij})_{x, 10^{-k}}|^{2} \, d \mu \leq C(n,C_{M})\,  \int_{0}^{1} \left(\,\dashint_{B_{r}(x)} |a_{ij}(y) - (a_{ij})_{x,r}|^{2} \, d \mu \right) \frac{dr}{r}. \ee

Since (\ref{l120}) is true for every $ i, \, j \in \{ 1, \ldots n+d\}$, we can take the sum over $i$ and $j$ on both sides of (\ref{l120}), and using (\ref{a'}) and (\ref{l11}), we get
\ben \sum_{k=1}^{\infty} \alpha^{2}(x, 10^{-k}) \leq C(n,C_{M}) \, \epsilon^{2} ,\een which is exactly (\ref{a}).

To prove inequality (\ref{a''}), fix $x \in M$ and $0 < r \leq \displaystyle \frac{1}{10}$. Then, there exists $k \geq 1$ such that \be \label{111}  10^{-k-1} < r \leq  10^{-k}, \quad \txt{that is} \quad \frac{1}{ 10^{-k}} \leq \frac{1}{r} < \frac{1}{ 10^{-k-1}}. \ee

Now, fix $i, \, j \in \{ 1, \ldots n+d\}$. Using inequality (\ref{10}) for $a = (a_{ij})_{x, 10^{-k}}$ and $r_{0} = r$, (\ref{111}), and the fact that $\mu$ is Ahlfors regular, we get that
\be \label{l121}\dashint_{B_{r}(x)} |a_{ij}(y) - (a_{ij})_{x,r}|^{2} \, d \mu \leq C(n,C_{M}) \, \dashint_{B_{ 10^{-k}}(x)} |a_{ij}(y) -  (a_{ij})_{x, 10^{-k}}|^{2} \, d \mu.\ee

Summing over $i$ and $j$ on both sides of (\ref{l121}), and using the definition of the the Frobenius norm together with (\ref{a'}), we get
\be \label{baaaa} \alpha^{2}(x,r) \leq C(n,C_{M}) \, \alpha^{2}(x,  10^{-k}). \ee
Taking the square root on both sides of (\ref{baaaa}) and using (\ref{a}) finishes the proof of (\ref{a''})
\end{proof}

Next, we use the Poincar\'e inequality to get good approximating $n$-planes for $M$ at every point $x \in M$ and at every scale $0 < r< \frac{1}{10 \lambda}$. In this context, a good approximating $n$-plane at the point $x \in M$ and radius $r$, is a plane $P_{x,r}$ such that the distance (in integral form) from $M \cap B_{r}(x)$ to $P_{x,r}$ is small.  

\begin{theorem}
\label{t1} Let $M \subset B_{2}(0)$ be an $n$-Ahlfors regular rectifiable set containing the origin, and let $\mu =$  \( \mathcal{H}^{n} \mres M\) be the Hausdorff measure restricted to $M$. Assume that $M$ satisfies the Poincar\'{e}-type inequality (\ref{eqp}).
There exists an $\epsilon_{1} >0 = \epsilon_{1}(n,d,C_{M})$, that for every $0 < \epsilon \leq \epsilon_{1}$, if
 \be \label{again} \int_{0}^{1} \left(\,\dashint_{B_{r}(x)} |\pi_{T_{y}M} - A_{x,r}|^{2} \, d \mu \right) \frac{dr}{r} < \epsilon^{2}, \quad \forall x \in M,\ee
then for every $x \in M$ and $0< r \leq \displaystyle \frac{1}{10 \lambda}$, there exists an affine $n$-dimensional plane $P_{x,r}$ such that 
\be \label{121} \dashint_{B_{r}(x)} \frac{d(y,P_{x,r})}{r} \, d \mu(y) \leq C \, \alpha(x,\lambda r),\ee
where $C= C(n,d,C_{P})$.
\end{theorem}

\begin{proof}
Fix $x \in M$ and $r \leq \displaystyle \frac{1}{10 \lambda}$. Let $\epsilon \leq \epsilon_{1}$ (with $\epsilon_{1}$ to be determined later) such that (\ref{again}) is satisfied. By (\ref{a'}), (\ref{a''}) from Lemma \ref{l5}, and the fact that $\lambda r \leq \displaystyle \frac{1}{10}$, we have 
\be \label{t11}  \dashint_{B_{\lambda r}(x)} |\pi_{T_{y}M} - A_{x,\lambda r}|^{2} \, d \mu = \alpha^{2} (x, \lambda r) \leq C(n,C_{M}) \,\epsilon^{2}.\ee
From (\ref{t11}) and the fact that $M$ is rectifiable (so approximate tangent planes exist $\mu$-a.e.), it is easy to check that there exists $y_{0} \in B_{\lambda r}(x) \cap M$ such that $T_{y_{0}}M$ exists, and  
\ben |\pi_{T_{y_{0}}M} - A_{x,\lambda r}| \leq \alpha(x,\lambda r) \leq C_{1} \, \epsilon,\een
where $C_{1}$ is a (fixed) constant depending only on $n$ and $C_{M}$. Comparing the operator norm with the Frobenius norm (the operator norm is at most the Frobenius norm), we get
\be \label{t13} ||\pi_{T_{y_{0}}M} - A_{x,\lambda r}|| \leq \alpha(x,\lambda r) \leq C_{1} \, \epsilon \leq C_{1} \epsilon_{1}.\ee

\noindent Let $\delta_{0}$ be the constant from Lemma \ref{l4}, and choose $\epsilon_{1} \leq \displaystyle\frac{\delta_{0}}{C_{1}}$. Then, (\ref{t13}) becomes 
\ben ||\pi_{T_{y_{0}}M} - A_{x,\lambda r}|| \leq \alpha(x,\lambda r) \leq \delta_{0}, \een and by Lemma \ref{l4} (with $\delta = \alpha(x,\lambda r)$, $V = T_{y_{0}}M$, and $L = A_{x,\lambda r}$), we deduce that $A_{x,\lambda r}$ has exactly $n$ eigenvalues such that $\lambda^{1}_{x,\lambda r}, \ldots, \lambda^{n}_{x,\lambda r}$ such that 
$ |\lambda^{i}_{x,\lambda r}| \geq 1 - c \, \alpha(x,\lambda r)$, for all $i \in \{ 1, \ldots , n\}$, and exactly $d$ eigenvalues $\lambda^{n+1}_{x,\lambda r}, \ldots, \lambda^{n+d}_{x,\lambda r}$ such that 
\be \label{eigen} |\lambda^{i}_{x,\lambda r}| \leq C(n,d) \, \alpha(x,\lambda r) \quad \forall \, i \in \{ n+1, \ldots , n+d\} .\ee

Since $A_{x,\lambda r}$ is a real symmetric matrix, $n+d$ eigenvectors of the matrix $A_{x,\lambda r}$, say $v^{1}_{x,\lambda r} , \ldots v^{n+d}_{x,\lambda r}$ (each corresponding to exactly one of the $n+d$ eigenvalues mentioned above) can be chosen to be orthonormal. Thus, $v^{1}_{x,\lambda r} , \ldots v^{n+d}_{x,\lambda r}$ are unit, linearly independent vectors such that
\be \label{value} A_{x,\lambda r}v^{i}_{x,\lambda r} = \lambda^{i}_{x,\lambda r}v^{i}_{x,\lambda r} \quad \forall \, i \in \{ 1, \ldots n+d\} .\ee
\\

\noindent Let us now fix our attention to the last $d$ eigenvector and eigenvalues. For $i \in \{ n+1 , \ldots n+d \}$ and consider the function $f_{i}$ on $\mathbb{R}^{n+d}$ defined by 
\ben f_{i}(y) = \left<y, v^{i}_{x,\lambda r}\right>, \,\,\,\,\,\,  y \in \mathbb{R}^{n+d}.\een
Notice that $f_{i}$ is a smooth function on $\mathbb{R}^{n+d}$, and for every point $y \in M$ where the tangent plane $T_{y}M$ exists, (which, again, is almost everywhere in $M$), we have \be \label{1}|\nabla^{M}f_{i}(y)| \leq  | \pi_{T_{y}M} - A_{x,\lambda r}| + |\lambda^{i}_{x,\lambda r}| . \ee 

In fact,
\ben \nabla^{M}f_{i}(y) = \pi_{T_{y}M} \big(\nabla f(y)\big) = \pi_{T_{y}M}(v^{i}_{x,\lambda r}) = (\pi_{T_{y}M} - A_{x,\lambda r})(v^{i}_{x,\lambda r}) + A_{x,\lambda r}v^{i}_{x,\lambda r}. \een
Thus, using the definition of the operator norm, the fact that $v^{i}_{x,\lambda r}$ is unit, (\ref{value}), and the fact that the operator norm of a matrix is at most its Frobenius norm we get 
\besn | \nabla^{M}f_{i}(y)| &\leq& |\pi_{T_{y}M} - A_{x,\lambda r})(v^{i}_{x,r})| + |A_{x,\lambda r}v^{i}_{x,\lambda r}| \\ &\leq&  ||\pi_{T_{y}M} - A_{x,\lambda r}|| + |\lambda^{i}_{x,\lambda r}| \leq |\pi_{T_{y}M} - A_{x,\lambda r}| + |\lambda^{i}_{x,\lambda r}| . \eesn

\noindent Now, applying the Poincar\'{e} inequality to the function $f_{i}$ and the ball $B_{r}(x)$, and using (\ref{1}), we get 
\be \label{2} \begin{split} \frac{1}{r} \, \dashint_{B_{r}(x)}  \left| \left<y,v^{i}_{x,\lambda r}\right> - \dashint_{B_{r}(x)} \left<z ,v^{i}_{x,\lambda r}\right> \right. & \left. d \mu(z) \right. \bigg| d \mu(y) \\ &\leq C_{P} \left(\, \dashint_{B_{\lambda r}(x)} \left( |\pi_{T_{y}M} - A_{x,\lambda r}| + |\lambda^{i}_{x,\lambda r}| \right)^{2} d \mu(y)\right)^{\frac{1}{2}}. \end{split} \ee

But $v^{i}_{x, \lambda r}$ is a constant vector, so (\ref{2}) can be rewritten as
\be \label{3} \begin{split} \frac{1}{r} \,\dashint_{B_{r}(x)}  \left| \left<y,v^{i}_{x,\lambda r}\right> - \left< \dashint_{B_{r}(x)} z d \mu(z) \right. \right. &,\left. \left. v^{i}_{x,\lambda r} \right. \bigg> \right| d \mu(y) \\ &\leq C_{P} \left( \, \dashint_{B_{\lambda r}(x)} \left( |\pi_{T_{y}M} - A_{x,\lambda r}| + |\lambda^{i}_{x,\lambda r}| \right)^{2} d \mu(y)\right)^{\frac{1}{2}}, \end{split} \ee

that is,
\be \label{4} \begin{split} \frac{1}{r} \, \dashint_{B_{r}(x)} \left|\left<y - \, \dashint_{B_{r}(x)} z \, d\mu(z) ,v^{i}_{x,\lambda r}\right> \right| &d \mu(y) \\ &\leq C_{P} \left( \,\dashint_{B_{\lambda r}(x)} \left( |\pi_{T_{y}M} - A_{x,\lambda r}| + |\lambda^{i}_{x,\lambda r}| \right)^{2} d \mu(y)\right)^{\frac{1}{2}} \\ 
&\leq  C(C_{P}) \, \left(  \left( \dashint_{B_{\lambda r}(x)}  |\pi_{T_{y}M} - A_{x,\lambda r}|^{2}\right)^{\frac{1}{2}} + |\lambda^{i}_{x,\lambda r}| \right). \end{split} \ee

\noindent Using (\ref{eigen}) and (\ref{a'}), (\ref{4}) becomes 
\be \label{t15} \frac{1}{r} \left(\, \dashint_{B_{r}(x)} \left|\left<y - \, \dashint_{B_{r}(x)} z \, d\mu(z) ,v^{i}_{x,\lambda r}\right> \right| d \mu(y) \right) \leq C(n,d,C_{P}) \,  \left( \dashint_{B_{\lambda r}(x)}  |\pi_{T_{y}M} - A_{x,\lambda r}|^{2}\right)^{\frac{1}{2}}.  \ee

\noindent Since (\ref{t15}) is true for every $ i \in \{n+1, \ldots , n+d \}$, we can take the sum over $i$ on both sides of (\ref{t15}) to get 

\be \label{t16} \frac{1}{r} \sum_{i=n+1}^{n+d} \dashint_{B_{r}(x)} \left|\left<y - \dashint_{B_{r}(x)} z d\mu(z) ,v^{i}_{x,\lambda r}\right> \right| d \mu(y) \leq C(n,d,C_{P})  \left( \dashint_{B_{\lambda r}(x)}  |\pi_{T_{y}M} - A_{x,\lambda r}|^{2}\right)^{\frac{1}{2}}. \ee

\noindent We are now ready to choose our plane $P_{x,r}$. Take $P_{x,r}$ to be the $n$-plane passing through the point $ c_{x,r} := \,\dashint_{B_{r}(x)} z \, d \mu(z) $, the centre of mass of $\mu$ in the ball $B_{r}(x)$, and such that $P_{x,r} - c = \txt{span} \{ v^{1}_{x,\lambda r}, \ldots , v^{n}_{x,\lambda r} \}$. In other words, $(P_{x,r} - c_{x,r})^{\perp} = \txt{span} \{ v^{n+1}_{x,\lambda r}, \ldots , v^{n+d}_{x,\lambda r} \}$. Here $(P_{x,r} - c_{x,r})^{\perp}$ denotes the $d$-plane of $\mathbb{R}^{n+d}$ perpendicular to the $n$-plane $P_{x,r} - c_{x,r}$. \\
\\

\noindent For $y \in B_{r}(x)$, we have that
\be \label{5} d(y, P_{x,r}) = d(y - c_{x,r}, P_{x,r} - c_{x,r}) = \left| \sum_{i=n+1}^{n+d} \left< y - c_{x,r} , v^{i}_{x,\lambda r}\right> v^{i}_{x,\lambda r}\right| \leq \sum_{i=n+1}^{n+d} \left|  \left< y - c_{x,r} , v^{i}_{x,\lambda r} \right>\right|
\ee

\noindent Dividing by $r$ and taking the average over $B_{r}(x)$ on both sides of (\ref{5}), and using the definition of $c_{x,r}$, we get 

\besn  \dashint_{B_{r}(x)} \frac{d(y, P_{x,r})}{r} \, d \mu(y)  &\leq & \, \frac{1}{r} \sum_{i=n+1}^{n+d} \dashint_{B_{r}(x)}  \left|  \left< y - \,\dashint_{B_{r}(x)} z \, d \mu(z)  , v^{i}_{x,\lambda r} \right>\right|\, d \mu(y) \\
& \leq &  C(n,d,C_{P}) \, \left(\, \dashint_{B_{\lambda r}(x)} \left| \pi_{T_{y}M} - A_{x,\lambda r} \right|^{2} d \mu\right)^{\frac{1}{2}},   \eesn
where the last inequality comes from (\ref{t16}).\\

\noindent Thus, by the definition of $\alpha(x,\lambda r)$ (see (\ref{a'})), we get (\ref{121}) and the proof is done.
\end{proof}

As mentioned earlier, we want to use the construction of the bi-Lipschitz map given by David and Toro in their paper \cite{DT1}. To do that, we introduce the notion of a \textbf{coherent collection of balls and planes}. Here, we follow the steps given by David and Toro (see \cite{DT1}, chapter 2). \\

First, let $l_{0} \in \mathbb{N}$ such that $10^{l_{0}} \leq \lambda \leq 10^{l_{0}+1}$, and set $r_{k} = 10^{-k-l_{0} - 5}$ for $ k \in \mathbb{N}$, and let $\epsilon$ be a small number (will be chosen later) that depends only on $n$ and $d$. Choose a collection $\{x_{jk} \}, \, \, j \in J_{k}$ of points in $\mathbb{R}^{n+d}$, so that

\be \label{59} |x_{jk} - x_{ik}| \geq r_{k} \quad \txt{for}\,\,\, i,j \in J_{k}, \, i \neq j. \ee

Set $B_{jk} := B_{r_{k}}(x_{jk})$ and $V_{k}^{\lambda} := \displaystyle \bigcup_{j \in J_{k}} \lambda B_{jk} =  \displaystyle \bigcup_{j \in J_{k}}  B_{\lambda r_{k}}(x_{jk}),\,$ for $\lambda > 1$.\\

We also ask for our collection  $\{x_{jk} \}, \, \, j \in J_{k}$ and $k \geq 1$ to satisfy
\be \label{60} x_{jk} \in V_{k-1}^{2} \quad \txt{for}\,\,\, k \geq 1 \,\,\, \txt{and} \,\,\, j \in J_{k}. \ee

Suppose that our initial net $\{x_{j0} \}$ is close to an $n$-dimensional plane $\Sigma_{0}$, that is
\be \label{101} d(x_{j0},\Sigma_{0}) \leq \epsilon \quad \forall\, j \in J_{0}.\ee

For each $k \geq 0$ and $j \in J_{k}$, suppose you have an $n$-dimensional plane $P_{jk}$, passing through $x_{jk}$ such that the following compatibility conditions hold:\\

\be \label{102} d_{x_{i0},100r_{0}}(P_{i0}, \Sigma_{0}) \leq \epsilon \,\,\,\, \txt{for} \,\, i \in J_{0}, \ee

\be \label{74} d_{x_{ik},100r_{k}}(P_{ik},P_{jk}) \leq \epsilon \,\,\,\, \txt{for}\, k\geq 0 \,\,\,\txt{and} \,\,\, i,j \in J_{k} \,\,\, \txt{such that} \,\,\, |x_{ik} - x_{jk}| \leq 100r_{k},\ee

and

\be \label{75} d_{x_{ik},20r_{k}}(P_{ik},P_{j,k+1}) \leq \epsilon \,\,\, \txt{for}\, k\geq 0 \,\,\,\txt{and} \,\,\, i \in J_{k},\,j \in J_{k+1} \,\, \txt{such that} \,\,\, |x_{ik} - x_{j,k+1}| \leq 2r_{k}.\ee

We can now define a \textbf{coherent collection of balls and planes}:

\begin{definition}
A \textbf{coherent collection of balls and planes}, (in short a CCBP), is a triple $(\Sigma_{0}, \{B_{jk} \}, \{P_{jk}\})$ where the properties (\ref{59}) up to (\ref{75}) above are satisfied, with a prescribed $\epsilon$ that is small enough, and depends only on $n$ and $d$.
\end{definition}

\begin{theorem} \label{t2}  (see Theorems 2.4 in \cite{DT1})
There exists $\epsilon_{2} > 0$ depending only on $n$ and $d$, such that the following holds: If $\epsilon \leq \epsilon_{2}$, and $(\Sigma_{0}, \{B_{jk} \}, \{P_{jk}\})$ is a CCBP (with $\epsilon$), then there exists a bijection $g: \mathbb{R}^{n+d} \rightarrow \mathbb{R}^{n+d}$ with the following properties:
\be g(z)= z \quad \txt{when} \,\,\, d(z, \Sigma_{0}) \geq 2, \ee
and
\be |g(z)-z| \leq C^{'}_{0} \epsilon \quad \txt{for} \,\,\, z \in \mathbb{R}^{n+d}, \ee
where $C^{'}_{0}= C^{'}_{0}(n,d)$. Moreover, $g(\Sigma_{0})$ is a $C^{'}_{0} \epsilon$-Reifenberg flat set that contains the accumulation set 
\besn E_{\infty} = &\{x&  \in \mathbb{R}^{n+d}; \,\, x\,\, \txt{can be written as} \\ &x& = \lim_{m \to \infty} x_{j(m),k(m)}, \,\, \txt{with}\,\,k(m) \in \mathbb{N}, \\ &\txt{and}& \,\, j(m) \in J_{k_{m}} \,\, \txt{for}\,\, m \geq 0 \,\, \txt{and} \,\, \lim_{m \to \infty} k(m) = \infty\} .\eesn
\end{theorem}

In \cite{DT1}, David and Toro give a sufficient condition for $g$ to be bi-Lipschitz that we want to use in our proof. To state this condition, we need some technical details from the construction of the map $g$ from Theorem \ref{t2}. So, let us briefly discuss the construction here: David and Toro defined a mapping $f$ whose goal is to push a small neighborhood of $\Sigma_{0}$ towards a final set, which they proved to be Reifenberg flat. They obtained $f$ as a limit of the composed functions $f_{k} = \sigma_{k-1} \circ \ldots \sigma_{0}$
where each $\sigma_{k}$ is a smooth function that moves points near the planes $P_{jk}$ at the scale $r_{k}$. More precisely, 
\be \label{sigmas} \sigma_{k} (y) = y + \sum_{j \in J_{k}} \theta_{jk}(y)[\pi_{jk}(y)-y],\ee
where $\{\theta_{jk}\}_{j \in J_{k}, k\geq 0}$ is a partition of unity with each $\theta_{jk}$ supported on $10B_{jk}$, and $\pi_{jk}$ denotes the orthogonal projection from $\mathbb{R}^{n+d}$ onto the plane $P_{jk}$.\\

\noindent Since $f$ in their construction was defined on $\Sigma_{0}$, $g$ was defined to be the extension of $f$ on the whole space.\\
 
\begin{corollary} \label{cr1} (see Proposition 11.2 in \cite{DT1}) Suppose we are in the setting of Theorem \ref{t2}. Define the quantity
\be  \label{nasty} \begin{split} \epsilon^{'}_{k}(y) &= \\ &sup\{d_{x_{im},100 r_{m}}(P_{jk},P_{im}); \,\,\, j \in J_{k}, \,\, i \in J_{m},\,\,\, m \in \{k,k-1\}, \,\,\txt{and}\,\, y \in 10B_{jk} \cap 11B_{im} \} \end{split} \ee
for $k \geq 1 \,\, \txt{and} \,\,y \in V_{k}^{10}$, and $\epsilon_{k}^{'}(y)=0 \,\, \txt{when}\,\, y \in \mathbb{R}^{n+d} \setminus V_{k}^{10}$ (when there are no pairs $(j,k)$ as above). 
If there exists $N > 0$ such that
 \be \label{89} \sum_{k=0}^{\infty} \epsilon^{'}_{k}(f_{k}(z))^{2} < N,  \ee
 then the map $g$ constructed in Theorem \ref{t2} is $K$-bi-Lipschitz, where the bi-Lipschitz constant $K = K(n,d,N)$.\end{corollary}

We are finally ready to prove Theorem \ref{MTT'}.

\textbf{\underline{\textit{Proof of Theorem \ref{MTT'}:}}}

\begin{proof}
As mentioned before, from here on, the proof of this theorem is essentially the same as that of its co-dimension 1 analogue found in \cite{M1} (Theorem 1.5 in \cite{M1}). In fact, the essential differences in the proofs of Theorem \ref{MTT'} and its co-dimension 1 analogue took place in Lemma \ref{l5} and Theorem \ref{t1}. Thus, we continue this proof by outlining the main ideas and referring the reader to the proof of Theorem 1.5 in \cite{M1} for a more detailed proof.\\

Let $\epsilon_{0} > 0$ (to be determined later), and suppose that (\ref{103}) holds. Let $\epsilon_{2}$ be the constant from Theorem \ref{t2}. We would like to apply Theorem \ref{t2} for $\epsilon = \epsilon_{2}$, and then Corollary \ref{cr1}. So our first goal is to construct a CCBP, and we do that in several steps:\\
 Let us start with a collection $\{\tilde{x}_{jk}\},\, j \in J_{k}$ of points in $M \cap B_{\frac{1}{10^{l_{0} +4}}}(0)$ that is maximal under the constraint \be \label{62}|\tilde{x}_{jk} - \tilde{x}_{ik}| \geq \displaystyle\frac {4r_{k}}{3} \quad \txt{when}\,\, i,j \in J_{k}\,\,\, \txt{and}\,\,\,i \neq j.\ee
Of course, we can arrange matters so that the point $0$ belongs to our initial maximal set, at scale $r_{0}$. Thus, $0 = \tilde{x}_{i_{0},0} $ for some $i_{0} \in J_{0}$. Notice that for every $k \geq 0$, we have 
\be \label{63} M  \cap B_{\frac{1}{10^{l_{0} +4}}}(0)\subset \displaystyle \bigcup_{j \in J_{k}}\bar{B}_{\frac{4r_{k}}{3}}(\tilde{x}_{jk}).\ee
\\
Later, we choose \be \label{61} x_{jk} \in M \cap B_{ \frac{r_{k}}{6}}(\tilde{x}_{jk}), \quad j \in J_{k} .\ee

By (\ref{63}) and (\ref{61}), we can see

\be \label{63'} M  \cap B_{\frac{1}{10^{l_{0} +4}}}(0)\subset \displaystyle \bigcup_{j \in J_{k}}\bar{B}_{\frac{4r_{k}}{3}}(\tilde{x}_{jk}) \subset \displaystyle \bigcup_{j \in J_{k}}B_{\frac{3r_{k}}{2}}(x_{jk}) .\ee

 Using (\ref{62}), (\ref{61}), and (\ref{63'}), it is easy to see that the collection $\{x_{jk} \},\,\,\, j \in J_{k}$ satisfies (\ref{59}) and (\ref{60}). (for details, see \cite{M1}, page 23).\\

Next, we choose our planes $P_{jk}$ and our collection $\{ x_{jk} \}$, for $k \geq 0$ and $ j \in J_{k}$. Fix $k\geq 0$ and $j \in J_{k}$. Let $\epsilon_{1}$ be the constant from Theorem \ref{t1}. For
\be \label{c} \epsilon_{0} \leq \epsilon_{1}, \ee
we apply Theorem \ref{t1} to the point $\tilde{x}_{jk}$ (by construction  $\tilde{x}_{jk} \in M$) and radius $120r_{k}$ (notice that $120\,r_{k} \leq \frac{1}{10 \lambda}$) to get an $n$-plane $P_{\tilde{x}_{jk},120r_{k}}$, denoted in this proof by $ P^{'}_{jk}$ for simplicity reasons, such that 
 
\be \label{66} \dashint_{B_{120r_{k}}(\tilde{x}_{jk})} \frac{d(y, P^{'}_{jk} )}{120r_{k}} \, d \mu  \leq  \, C(n,d,C_{P}) \, \alpha(\tilde{x}_{jk}, 120 \lambda r_{k}). \ee

Thus, by (\ref{66}) and the fact that $\mu$ is Ahlfors regular, there exists $x_{jk} \in M \cap B_{ \frac{r_{k}}{6}}(\tilde{x}_{jk})$ such that 

\bes \label{68} d(x_{jk},P^{'}_{jk}) &\leq& \dashint_{B_{ \frac{r_{k}}{6}}(\tilde{x}_{jk})}d(y, P^{'}_{jk} ) \, d \mu \nonumber \\ &\leq&  C(n,C_{M}) \, \dashint_{B_{120r_{k}}(\tilde{x}_{jk})} d(y, P^{'}_{jk} ) \, d \mu \leq \, C(n,d,C_{M},C_{P}) \, \alpha(\tilde{x}_{jk}, 120 \lambda r_{k}) \,r_{k}. \ees

Let $P_{jk}$ be the plane parallel to $P^{'}_{jk}$ and passing through $x_{jk}$. From (\ref{66}), (\ref{68}) and the fact that the two planes are parallel, we see that (see \cite{M1} p. 24)

\be \label{73} \dashint_{ B_{120r_{k}}(\tilde{x}_{jk})} \frac{d(y,P_{jk})}{120r_{k}} \, d \mu \leq  C(n,d,C_{M},C_{P})\,  \alpha(\tilde{x}_{jk}, 120 \lambda r_{k}). \ee

To summarize what we did so far, we have chosen $n$-dimensional planes $P_{jk}$ for $k\geq 0$ and $j \in J_{k}$ where each $P_{jk}$ passes through $x_{jk}$, and satisfies (\ref{73}). Notice that (\ref{73}) shows that $P_{jk}$ is a good approximating plane to $M$ in the ball $B_{120r_{k}}(\tilde{x}_{jk})$.\\ 

We want to get our CCBP with $\epsilon_{2}$. Thus, we show that (\ref{101}), (\ref{102}), (\ref{74}), and (\ref{75}) hold with $\epsilon = \epsilon_{2}$. Since the proofs of these inequalities are the same as the proofs of their analogue inequalities in the co-dimension 1 case, we only outline their proofs here (see \cite{M1} p. 25-- p. 31 for a detailed proof of the inequalities).\\

\textbf{\textit{Outline of the proofs for (\ref{74}) and (\ref{75}):}} 

Inequalities  (\ref{74}) and (\ref{75}) can be proved simultaneously. Fix $k \geq 0$ and $j \in J_{k}$; let $m \in \{k, k-1 \}$ and $i \in J_m$ such that $|x_{jk} - x_{im}| \leq 100r_{m}$. We want to show that $P_{jk}$ and $P_{im}$ are close together. To do that, we construct $n$ linearly independent vectors that ``effectively'' span $P_{jk}$, (that is, these vectors span $P_{jk}$, and are far away from each other in a uniform quantitative manner), and that are close to $P_{im}$. More precisely, using Lemma \ref{l1} inductively, together with (\ref{73}), we can prove the following claim:\\

 \textbf{\textit{Claim 1:}} 
Denote by  $\pi_{jk}$ is the orthogonal projection of $\mathbb{R}^{n+d}$ on the plane $P_{jk}$. Let $r = c_{0} \, r_{k}$, where $c_{0} \leq \frac{1}{2}$ is the constant from Lemma \ref{l1} depending only on $n$, $d$, and $C_{M}$. There exists $C_{1} = C_{1}(n,d,C_{M},C_{P})$, such that if $C_{1} \epsilon_{0} \leq 1$, then there exists a sequence of $n+1$ balls $\{B_{r}(y_{l})\}_{l=0}^{n}$,  such that 
\begin{enumerate}
\item $\forall \, l \in \{ 0, \ldots n\}$, we have $y_{l} \in M$ and $B_{r}(y_{l}) \subset B_{2r_{k}}(\tilde{x}_{jk}).$
\item  $q_{1} - q_{0} \notin B_{5r}(0)$, and $\forall \,  l \in \{ 2, \ldots n\}$, we have $q_{l} - q_{0} \notin N_{5r}\big(span \{q_{1} - q_{0}, \ldots, q_{l-1} - q_{0}  \}\big),$
\end{enumerate}
where $q_{l} = \pi_{jk}(p(y_{l}))$ and $p(y_{l}) = \dashint_{B_{r}(y_{l})}z \, d\mu(z)$ is the centre of mass of $\mu$ in the ball $B_{r}(y_{l})$.\\

Now, on one hand, notice that 
\be \label{78} P_{jk} - q_{0} = span  \{q_{1} - q_{0} , \ldots, q_{n} - q_{0}  \}.\ee

On the other hand, by the definition of $p(y_{l})$, Jensen's inequality applied on the convex function $\phi(.) = d(.,P_{jk})$, the fact that $\mu$ is Ahlfors regular, $B_{r}(y_{l}) \subset B_{2r_{k}}(\tilde{x}_{jk})$, $r = c_{0}\, r_{k}$, and (\ref{73}), we have that 

\be \label{33} d\big(p(y_{l}),P_{jk}\big)
\leq  C(n,d, C_{M}, C_{P}) \, \alpha(\tilde{x}_{jk},120 \lambda  \,r_{k}) \, r_{k}, \quad \forall  \, l \in \{ 0, \ldots n\}.\ee

Similarly, we have that 
\be \label{33again} d\big(p(y_{l}),P_{im}\big)
\leq  C(n,d, C_{M}, C_{P}) \, \alpha(\tilde{x}_{im},120 \lambda  \,r_{m}) \, r_{m},  \quad \forall \,  l \in \{ 0, \ldots n\}. \ee

Thus, combining (\ref{33}) and (\ref{33again}), we directly get

\be \label{80} d\big(q_{l},P_{im}\big) \leq C(n,d,C_{M},C_{P}) \, \big( \alpha(\tilde{x}_{jk},120 \lambda r_{k})\,r_{k} + \alpha(\tilde{x}_{im}, 120 \lambda r_{m}) \, r_{m}\big),  \quad \forall  \, l \in \{ 0, \ldots n\}.
\ee

To compute the distance between $P_{jk}$ and $P_{im}$, let $y \in P_{jk} \cap B_{\rho}(x_{im})$ where $\rho= \{20r_m, 100r_{m}\}$. By (\ref{78}), $y$ can be written uniquely as
\be \label{82} y  = q_{0} + \sum_{l=1}^{n} \beta_{l}(q_{l} - q_{0}).\ee 

Using Lemma \ref{l3}, for $u_{l} = q_{l} - q_{0}$, $R = r$, and $v = y - q_{0}$ to get an upper bound on the $\beta_{l}$'s that show up in (\ref{82}), together with (\ref{80}), we get that 

\ben  d\big(y, P_{im}\big) \leq C(n,d,C_{M},C_{P}) \bigg( \alpha(\tilde{x}_{jk}, 120\lambda r_{k})\,r_{k} + \alpha(\tilde{x}_{im}, 120 \lambda r_{m}) \, r_{m}\bigg)
\een 

Thus, 
\be \label{88}d_{x_{im}, \rho} (P_{jk}, P_{im}) \leq  c \, \bigg( \alpha(\tilde{x}_{jk},120 \lambda r_{k}) + \alpha(\tilde{x}_{im},120 \lambda r_{m})\bigg)\,\,\,\,\, \rho \in\{20r_m, 100r_{m}\}. \ee

Now, by Lemma \ref{l5}, we know that
$ \alpha(\tilde{x}_{jk},120 \lambda r_{k}) \leq C(n,C_{M}) \, \epsilon_{0}$, and $\alpha(\tilde{x}_{im},120 \lambda r_{m}) \leq C(n,C_{M}) \, \epsilon_{0}$. Thus, (\ref{88}) becomes 
\be \label{135} d_{x_{im}, \rho} (P_{jk}, P_{im}) \leq  C(n,d,C_{M},C_{P}) \epsilon_{0} \,\,\,\,\, \rho \in\{20r_m, 100r_{m}\}. \ee

So, we have shown that there exist two constants $C_{2}$ and $C_{3}$, each depending only on $n$, $d$, $C_{M}$, and $C_{P}$, such that
\be \label{74'} d_{x_{ik},100r_{k}}(P_{ik},P_{jk}) \leq C_{2} \,\epsilon_{0} \,\,\,\, \txt{for}\, k\geq 0 \,\,\,\txt{and} \,\,\, i,j \in J_{k} \,\,\, \txt{such that} \,\,\, |x_{ik} - x_{jk}| \leq 100r_{k},\ee

and

\be \label{75'} d_{x_{ik},20r_{k}}(P_{ik},P_{j,k+1}) \leq C_{3} \,\epsilon_{0} \,\,\, \txt{for}\, k\geq 0 \,\,\,\txt{and} \,\,\, i \in J_{k},\,j \in J_{k+1} \,\, \txt{such that} \,\,\, |x_{ik} - x_{j,k+1}| \leq 2r_{k}.\ee

For 
\be \label{c'''} C_{2} \,\epsilon_{0} \leq \epsilon_{2} \quad \txt{and} \quad C_{3} \,\epsilon_{0} \leq \epsilon_{2}, \ee
we get (\ref{74}) and (\ref{75}). \\

\textbf{\textit{Outline of the proofs for (\ref{101}) and (\ref{102}):}} 

We start with (\ref{102}). Recall that $0 = \tilde{x}_{i_{0},0}$ for some $i_{0} \in J_{0}$. Choose $\Sigma_{0}$ to be the plane $P_{i_{0},0}$ described above (recall that $P_{i_{0},0}$ passes through $x_{i_{0},0}$, where $r_{0} = 10^{-l_{0} - 5}$). Then, what we need to show is 

\be \label{102'} d_{x_{j0},100r_{0}}(P_{j0}, P_{i_{0},0}) \leq \epsilon_{2} \quad \txt{for} \,\, j \in J_{0}. \ee

Fix $j \in J_{0}$, and take the corresponding $x_{j0}$. Since by construction $|\tilde{x}_{j0}| < \displaystyle \frac{1}{10^{l_{0}+4}}$ and since (\ref{61}) says that $ |x_{j_{0},0} -  \tilde{x}_{j_{0},0}| \leq \displaystyle \frac{r_{0}}{6}$, then, we have
\be \label{1''}|x_{j0}| \leq \frac{r_{0}}{6} + \frac{1}{10^{l_{0}+4}},\,\,\,\,\,\, j \in J_{0}.\ee
Moreover, by (\ref{61}) and the fact that $0 = \tilde{x}_{i_{0},0}$ , we have
\be \label{2''} |x_{i_{0},0} -  \tilde{x}_{i_{0},0}| = |x_{i_{0},0}| \leq \frac{r_{0}}{6}.\ee
Combining (\ref{1''}) and (\ref{2''}), and using the fact that $r_{0} = 10^{-l_{0}-4}$ we get
\be |x_{j0} - x_{i_{0},0} | \leq \frac{r_{0}}{6} +  \frac{1}{10^{l_{0}+4}} + \frac{r_{0}}{6} \leq \frac{r_{0}}{6} +  10 r_{0} + \frac{r_{0}}{6} \leq 100r_{0}. \ee
Thus, by (\ref{74}) for $x_{ik} = x_{j0}$, $P_{ik} = P_{j0}$, and $P_{jk} = P_{i_{0},0}$, we get exactly (\ref{102'}), hence finishing the proof for (\ref{102}).\\

It remains to show (\ref{101}) with $\epsilon = \epsilon_{2}$,  that is 
\be \label{101'} d(x_{j0}, P_{i_{0},0}) \leq \epsilon_{2} ,\quad \txt{for}\,\, j \in J_{0}. \ee
However, notice that since $x_{j0} \in P_{j0}$, (\ref{101}) follows directly from (\ref{102}).\\

We finally have our CCBP. Now, by the proof of Theorem \ref{t2} (see paragraph above (\ref{sigmas})) we get the smooth maps $\sigma_{k} \,\, \txt{and} \,\, f_{k} = \sigma_{k-1} \circ \ldots \sigma_{0} \,\, \txt{for} \,\, k \geq 0$, and then the map $f = \displaystyle \lim_{k \to \infty} f_{k}$ defined on $\Sigma_{0}$, and finally the map $g$ that we want.
 
Moreover, by Theorem \ref{t2}, we know that $g: \mathbb{R}^{n+d} \rightarrow \mathbb{R}^{n+d}$ is a bijection with the following properties:
\be \label{aa1} g(z)= z \,\,\, \,\,\, \txt{when} \,\,\, d(z, \Sigma_{0}) \geq 2, \ee
\be \label{bb1} |g(z)-z| \leq C^{'}_{0} \epsilon_{2} \,\,\,\,\,\,\, \txt{for} \,\,\, z \in \mathbb{R}^{n+d}, \ee
and \be \label{cc1} g(\Sigma_{0}) \,\, \txt{is a } \,\,\, C^{'}_{0} \epsilon_{2} \txt{-Reifenberg flat set}. \ee 

Fix $\epsilon_{0}$ such that (\ref{c}), (\ref{c'''}), and the hypothesis of Claim 1 are all satisfied. Notice that by the choice of $\epsilon_{0}$, we can write 
 $\epsilon_{0} = c_{4}\, \epsilon_{2}$, 
where $c_{4} = c_{4}(n,d,C_{M},C_{P})$. Hence, from (\ref{aa1}), (\ref{bb1}), (\ref{cc1}), we directly get (\ref{aa}), (\ref{bb}), and (\ref{cc}).
\\

Next, we show that 
\be \label{12''} M \cap B_{\frac{1}{10^{l_{0}+4}}}(0) \subset  g(\Sigma_{0}). \ee Fix $x \in  M \cap B_{\frac{1}{10^{l_{0}+4}}}(0) $. Then, by (\ref{63'}), we see that for all $k \geq 0$, there exists a point $x_{jk}$ such that $|x - x_{jk}| \leq \displaystyle \frac{3r_{k}}{2}$, and hence $x \in E_{\infty} \subset g(\Sigma_{0})$ ($E_{\infty}$ is the set defined in Theorem \ref{t2}). Since $x$ was an arbitrary point in $M \cap B_{\frac{1}{10^{l_{0}+4}}}(0)$, (\ref{12''}) is proved. This shows that (\ref{contained}) holds for $\theta_{0} := \frac{1}{10^{l_{0}+4}}$.\\

We still need to show that $g$ is bi-Lipschitz. By Corollary \ref{cr1}, it suffices to show (\ref{89}). To do that, we need the following inequality from \cite{DT1} (see inequality (6.8) page 27 in \cite{DT1} 

\be \label{in} |f(z) - f_{k}(z)| \leq C(n,d) \epsilon_{2} \, r_{k} \quad \txt{for}\,\, k \geq 0\,\, \txt{and} \,\, z \in \Sigma_{0}.\ee

Let $z \in \Sigma_{0}$, and choose $\bar{z} \in M \cap B_{\frac{1}{10^{l_{0}+4}}}(0)$ such that 
\be \label{90} |\bar{z} - f(z)| \leq 2 \, d(f(z),M \cap B_{\frac{1}{10^{l_{0}+4}}}(0)). \ee
Fix $k \geq 0$, and consider the index $m  \in \{k,k-1\}$ and the indices $j \in J_{k}$ and $i \in J_{m}$ such that $f_{k}(z) \in 10B_{jk} \cap 11B_{im}$. We show that
\be \label{91}  d_{x_{im},100 r_{m}}(P_{jk},P_{im}) \leq C(n,d,C_{M},C_{P}) \, \alpha(\bar{z},r_{k-l_{0}-5}) \quad \txt{for} \,\, k \geq 1. \ee

In fact, by (\ref{90}) and (\ref{in}), and since $\tilde{x}_{jk} \in M \cap B_{\frac{1}{10^{l_{0}+4}}}(0)$, $|\tilde{x}_{jk} - x_{jk}| \leq \displaystyle  \frac{r_{k}}{6}$, and $f_{k}(z) \in 10B_{jk}$, one can show that (see \cite{M1} p. 32-33 for detailed proof)

 \be \label{96} B_{120 \lambda r_{m}}(\tilde{x}_{im}) \cup  B_{120 \lambda r_{k}}(\tilde{x}_{jk}) \subset B_{r_{k-l_{0}-5}}(\bar{z}). \ee

Now, writing $\pi_{T_{y}M} = \big(a_{pq}(y)\big)_{pq}$, and using the definition of the Frobenius norm, together with (\ref{10}) for $ a = (a_{pq})_{\bar{z},r_{k-l_{0}-5}}$, (\ref{96}), and the fact that $\mu$ is Ahlfors regular

\besn \alpha^{2}(\tilde{x}_{jk}, 120 \lambda r_{k}) &=& \dashint_{B_{120 \lambda r_{k}}(\tilde{x}_{jk})} |\pi_{T_{y}M} - A_{\tilde{x}_{jk}, 120 \lambda r_{k}}|^{2}\, d \mu\\ 
&=&  \sum_{p,q = 1} ^{n+d} \dashint_{B_{120 \lambda r_{k}}(\tilde{x}_{jk})}  |a_{pq}(y) - (a_{pq})_{\tilde{x}_{jk}, 120 \lambda r_{k}}|^{2}\, d \mu \\
&\leq&  \sum_{p,q = 1} ^{n+d} \dashint_{B_{120 \lambda r_{k}}(\tilde{x}_{jk})}  |a_{pq}(y) - (a_{pq})_{\bar{z},r_{k-l_{0} -5}}|^{2}\, d \mu \\
&\leq& C(n,C_{M}) \sum_{p,q = 1} ^{n+d} \, \dashint_{B_{r_{k-l_{0}-5}}(\bar{z})}  |a_{pq}(y) - (a_{pq})_{\bar{z},r_{k-l_{0} -5}}|^{2}\, d \mu \\
&=& C(n,C_{M}) \dashint_{B_{r_{k-l_{0}-5}}(\bar{z})} |\pi_{T_{y}M} - A_{\bar{z},r_{k-l_{0}-5}}|^{2}\, d \mu \\
&=& C(n,C_{M})\, \alpha^{2}(\bar{z},r_{k-l_{0}-5}),
\eesn 
and thus,
\be \label{97}  \alpha(\tilde{x}_{jk}, 120 \lambda r_{k})  \leq C(n,C_{M})\,  \alpha(\bar{z},r_{k-l_{0}-5}). \ee

Similarly, we can show that 
\be \label{98}  \alpha(\tilde{x}_{im}, 120 \lambda r_{m})  \leq C(n,C_{M})\,  \alpha(\bar{z},r_{k-l_{0}-5}). \ee

Plugging (\ref{97}) and (\ref{98}) in (\ref{88}) for $\rho = 100r_{m}$, we get
\be \label{99}  d_{x_{im},100 r_{m}}(P_{jk},P_{im}) \leq C(n,d,C_{M},C_{P}) \, \alpha(\bar{z},r_{k-l_{0}-5}), \quad \forall k \geq 1. \ee

This finishes the proof of (\ref{91}).\\
\\
Hence, we have shown that 
$\epsilon^{'}_{k}(f_{k}(z)) \leq C(n,d,C_{M},C_{P}) \,  \alpha(\bar{z},r_{k-l_{0}-5})$ for every $k\geq 1$, that is
\be \label{136}  \epsilon^{'}_{k}(f_{k}(z))^{2} \leq C(n,d,C_{M},C_{P}) \,  \alpha^{2}(\bar{z},r_{k-l_{0}-5}), \quad \forall \, k\geq 1\ee
Summing both sides of (\ref{136}) over $k \geq 0$, and using (\ref{a}) in Lemma \ref{l5} together with the fact that $\bar{z} \in M\cap B_{\frac{1}{10^{l_{0}+4}}}(0)$, we get
\be \label{100} \sum_{k=0}^{\infty} \epsilon^{'}_{k}(f_{k}(z))^{2} \leq 1 + C(n,d,C_{M},C_{P}) \, \sum_{k=1}^{\infty} \alpha^{2}(\bar{z},r_{k-l_{0}-5}) \leq 1 + C(n,d,C_{M},C_{P}) \, \epsilon^{2}_{0} \,\, := N.
\ee
Inequality (\ref{89}) is proved, and our theorem follows.

\end{proof}

\vspace{0.5cm}
As mentioned in the introduction, in the special case when $M$ has co-dimension 1, (\ref{103}) translates a Carleson-type condition on the oscillation of the unit normals to $M$.\\

\textbf{\textit{Proof that Theorem \ref{Th1Merhej} follows from Theorem \ref{MTT'}}}

\begin{proof}
Suppose that (\ref{103old}) holds for some choice of unit normal $\nu$ to $M$. We show that (\ref{103old}) is in fact exactly inequality (\ref{103}). Fix $x \in M$ and $ 0 < r < 1$ and let $y \in M \cap B_{r}(x)$ be a point where the approximate tangent plane $T_{y}M$ (and thus the unit normal $\nu(y)$) exists. Denote by $ {T_{y}M}^{\perp} $ the subspace perpendicular to $T_{y}M$. Then, using the matrix representation of $ \pi_{T_{y}M} $ in the standard basis of $\mathbb{R}^{n+1}$, and the fact that $ \pi_{{T_{y}M}^{\perp}} = Id_{n+1} - \pi_{T_{y}M}$ where $Id_{n+1}$ is the $(n+1) \times (n+1)$ identity matrix,  one can easily see that
\be \label{r1} |\pi_{T_{y}M} - A_{x,r}|^{2} = |\pi_{{T_{y}M}^{\perp}}- B_{x,r}|^{2},  \ee 
where $ \pi_{{T_{y}M}^{\perp}} = \big ( b_{ij}(y) \big)_{ij}$ and $B_{x,r} = Id_{n+d} - A_{x,r} =  \big( (b_{ij})_{x,r} \big)_{ij}$.\\

Now, we want to express the right hand side of (\ref{r1}) using a different basis than the standard basis of $\mathbb{R}^{n+1}$. For any choice of orthonormal basis $\{ \nu_{1}(y), \ldots \nu_{n}(y) \}$ of $T_{y}M$, we have that $\{ \nu_{1}(y), \ldots , \nu_{n}(y), \nu(y)\}$ is an orthonormal basis for $\mathbb{R}^{n+1}$. The matrix representation of $ \pi_{{T_{y}M}^{\perp}}$ with $\{ \nu_{1}(y), \ldots , \nu_{n}(y), \nu(y)\}$ as a basis for the domain $\mathbb{R}^{n+1}$ and the standard basis for the range $\mathbb{R}^{n+1}$, is the $(n+1) \times (n+1)$ matrix whose last column is $\nu(y)$ while the other columns are all zero. Thus, with this choice of bases and matrix representations, $B_{x,r}$ becomes the matrix whose last column is $\nu_{x,r}$ while the other column are all zero  \footnote{ Note that considering this choice of bases and matrix representations is only valid in co-dimension 1, as otherwise $B_{x,r}$ will not be well defined. This is because in higher co-dimensions, one will have infinitely many choices for the unit normals that span the normal plane, instead of the one choice (modulo direction) in co-dimension 1.}. Hence, using (\ref{r1}), we get that
\be \label{cor1}  |\pi_{T_{y}M} - A_{x,r}|^{2} =  |\pi_{{T_{y}M}^{\perp}}- B_{x,r}|^{2} = | \nu(y) - \nu_{x,r}|^{2}. \ee
Since (\ref{cor1}) is true for any $y \in B_{r}(x)$, and since $x$ and $r$ are arbitrary, then, 
\ben \begin{split} \sup_{x \in M \cap B_{\frac{1}{10^{4}}}(0)} \,\, \int_{0}^{1} \left(\, \dashint_{B_{r}(x)} |\nu(y) - \nu_{x,r}|^{2} d \mu \right) & \frac{dr}{r} = \\ &\sup_{x \in M \cap B_{\frac{1}{10^{4}}}(0)} \,\, \int_{0}^{1} \left(\,\dashint_{B_{r}(x)}  |\pi_{T_{y}M} - A_{x,r}|^{2} d \mu \right) \frac{dr}{r} , \end{split} \een
and the proof is done
\end{proof}

We now show that if we assume, in addition to the hypothesis of Theorem \ref{MTT'}, that $M$ is Reifenberg flat, then (locally)  $M$ is exactly the bi-Lipschitz image of an $n$-plane. In other words, the containment in (\ref{contained}) becomes an equality. 

\begin{corollary} \label{MTT'RF}
Let $M \subset B_{2}(0)$ be an $n$-Ahlfors regular rectifiable set containing the origin, and let $\mu =$  \( \mathcal{H}^{n} \mres M\) be the Hausdorff measure restricted to $M$. Assume that $M$ satisfies the Poincar\'{e}-type inequality (\ref{eqp}). There exist $\epsilon_{3} = \epsilon_{3}(n,d,C_{M},C_{P})>0$, and $\theta_{1} = \theta_{1}(\lambda)$ such that if  (\ref{103}) is satisfied with $\epsilon_{3}$ instead of $\epsilon_{0}$, and if for every $x \in M$ and $r < 1$ there is an $n$-plane $Q_{x,r}$, passing through $x$ such that 
\be \label{rf1} d(y, Q_{x,r}) \leq \epsilon_{3} \, r \quad \forall \, y \in M \cap B_{10r}(x) \ee
and 
\be \label{rf2} d(y, M) \leq \epsilon_{3} \, r \quad \forall \, y \in Q_{x,r} \cap B_{10r}(x), \ee
then there exists an onto $K$-bi-Lipschitz map $g: \mathbb{R}^{n+d} \rightarrow \mathbb{R}^{n+d}$ where the bi-Lipschitz constant $K=K(n,d,C_{M},C_{P})$ and an $n$-dimensional plane $\Sigma_{0}$, such that (\ref{aa}) holds, (\ref{bb}) holds with $\epsilon_{3}$ instead of $\epsilon_{0}$, and with $C_{0}'' = C_{0}''(n,d,C_{M},C_{P})$ instead of $C_{0}$, and
\be \label{containedrf} M \cap B_{\theta_{1}}(0) =  g(\Sigma_{0}) \cap  B_{\theta_{1}}(0). \ee
\end{corollary}

\begin{proof}
Let $\epsilon_{2}$ be as in Theorem \ref{t2}, and let $\epsilon_{3} \leq \epsilon \leq \epsilon_{2}$ ($\epsilon_{3}$ and $\epsilon$ to be determined later).  Going through the exact same steps as in the proof of Theorem \ref{MTT'}, but with $\epsilon$ instead of $\epsilon_{2}$, and $\epsilon_{3}$ instead of $\epsilon_{0}$, we get a bijective map $g: \mathbb{R}^{n+d} \rightarrow \mathbb{R}^{n+d}$ such that (\ref{aa}) holds,
\be \label{bb1rf} |g(z)-z| \leq C^{'}_{0} \epsilon \quad, \txt{for} \,\,\, z \in \mathbb{R}^{n+d}, \ee
and \be \label{cc1rf} M \cap B_{\frac{1}{10^{l_{0}+4}}}(0) \subset g(\Sigma_{0}).  \ee 

Note that we have not fixed $\epsilon_{3}$ and $\epsilon$ yet. However, we know that the above holds for $\epsilon_{3} \leq \epsilon \leq \epsilon_{2}$ while inequality (\ref{c}) is satisfied with $\epsilon_{3}$ instead of $\epsilon_{0}$, (\ref{c'''}) is satisfied with $\epsilon$ instead of $\epsilon_{2}$ and $\epsilon_{3}$ instead of $\epsilon_{0}$, and the hypothesis of  Claim 1 is satisfied with $\epsilon_{3}$ instead of $\epsilon_{0}$. Now, we want to show that 
\be \label{rf3} g(\Sigma_{0}) \cap B_{\frac{1}{10^{l_{0}+8}}}(0) \subset M.  \ee
We first show that for every $k \geq 0$ and for every $j \in J_{k}$, $M \cap B_{120 r_{k}}(\tilde{x}_{jk})$ is close to $P_{jk}$ and that the $n$-planes $P_{jk}$ and $Q_{jk} := Q_{x_{jk}, r_{k}}$ are close to each other (in the Hausdorff distance sense). Let us begin by showing that for every $k \geq 0$ and for every $j \in J_{k}$,

\be \label{rf4} d(z, P_{jk}) \leq \epsilon \, r_{k} \quad \forall \, z \in M \cap B_{120 r_{k}}(\tilde{x}_{jk}).  \ee

By Markov's inequality, we know that
\ben \begin{split} \label{3''} \mu \bigg( x \in B_{120r_{k}}(\tilde{x}_{jk}); \frac{d(x,P_{jk})}{120r_{k}}\geq \alpha^{\frac{1}{2}} (\tilde{x}_{jk}, 120 \lambda r_{k}) & \bigg) \leq \\ &\frac{1}{ \alpha^{\frac{1}{2}}(\tilde{x}_{jk}, 120 \lambda r_{k})} \int_{B_{120r_{k}}(\tilde{x}_{jk})} \frac{d(y,P_{jk})}{120r_{k}} \, d \mu \end{split} \een

Using (\ref{73}) with the fact that $\mu$ is Ahlfors regular, and (\ref{103}) with (\ref{a''}) from Lemma \ref{l5} and the fact that $120 \lambda r_{k} \leq \frac{1}{10}$, we get 

\be \begin{split} \label{4''}   \mu \bigg( x \in B_{120r_{k}}(\tilde{x}_{jk}); \frac{d(x,P_{jk})}{120r_{k}}\geq \alpha^{\frac{1}{2}} (\tilde{x}_{jk}, 120 \lambda r_{k}) &\bigg) \leq \nonumber \\ & \frac{\mu(B_{120r_{k}}(\tilde{x}_{jk}))}{ \alpha^{\frac{1}{2}}(\tilde{x}_{jk}, 120 \lambda r_{k})} \, \dashint_{B_{120r_{k}}(\tilde{x}_{jk})} \frac{d(y,P_{jk})}{120r_{k}} \, d \mu  \nonumber \\
&\leq C(n,d, C_{M}, C_{P}) \, r_{k}^{n} \, \alpha^{\frac{1}{2}}\left(\tilde{x}_{jk},120 \lambda r_{k}\right) \nonumber \\ &\leq C(n,d, C_{M}, C_{P})\, r_{k}^{n} \, \epsilon_{3} ^{\frac{1}{2}}.
\end{split} \ee

Now, take a point $z \in M \cap B_{120r_{k}}(\tilde{x}_{jk})$. We consider two cases:\\
Either 
\be \label{5'''} \frac{d(z ,P_{jk})}{120r_{k}} \leq \alpha^{\frac{1}{2}}\left(\tilde{x}_{jk}, 120 \lambda r_{k}\right) \ee
or
\be \label{5'} \frac{d(z ,P_{jk})}{120r_{k}} > \alpha^{\frac{1}{2}}\left(\tilde{x}_{jk}, 120 \lambda r_{k}\right). \ee

In the first case, combining (\ref{5'''}) with (\ref{103}) and (\ref{a''}), we get 

\be \label{6} d(z ,P_{jk}) \leq C(n, C_{M}) \, r_{k} \, \epsilon_{3}^{\frac{1}{2}} .\ee

In case of (\ref{5'}), let $\rho$ be the biggest radius such that \ben B_{\rho}(z) \subset  \left\{ x \in B_{120r_{k}}(\tilde{x}_{jk}) ; \,\,\, \frac{d(x,P_{jk})}{120r_{k}} > \alpha^{\frac{1}{2}}\left(\tilde{x}_{jk}, 120 \lambda r_{k}\right)\right\}. \een

Now, since $z \in M$ and $\mu$ is Ahlfors regular, we get using (\ref{4''}) that

\be \label{8''} C_{M} \, \rho^{n} \leq \mu(B_{\rho}(z)) \leq C(n,d, C_{M}, C_{P}) \, r_{k}^{n} \,\epsilon_{3}^{\frac{1}{2}}. \ee

Thus, relabelling, (\ref{8''}) becomes

\be \label{9''} \rho \leq C(n, C_{M}, C_{P}) \,r_{k} \, \epsilon_{3}^{\frac{1}{2n}}.\ee

On the other hand, since $\rho$ is the biggest radius such that $B_{\rho}(z) \subset \\ \left\{ x \in B_{120r_{k}}(\tilde{x}_{jk}) ; \,\,\, \frac{d(x,P_{jk})}{120r_{k}} > \alpha^{\frac{1}{2}}\left(\tilde{x}_{jk}, 120 \lambda r_{k}\right)\right\}$ , then there exists $x_{0} \in \partial B_{\rho}(z)$ such that 
\be \label{5''} \frac{d(x_{0} ,P_{jk})}{120r_{k}} \leq \alpha^{\frac{1}{2}}\left(\tilde{x}_{jk}, 120 \lambda r_{k}\right). \ee

Thus, by (\ref{5''}), (\ref{9''}) and (\ref{103}) together with (\ref{a''}), we get

\bes \label{7''} d(z ,P_{jk}) &\leq& |z-x_{0}| +  d(x_{0} ,P_{jk}) \nonumber \\
&=& \rho +  d(x_{0} ,P_{jk}) \leq  C(n,d, C_{M}, C_{P}) \, r_{k} \, \epsilon_{3}^{\frac{1}{2n}}+ 120r_{k} \, \alpha^{\frac{1}{2}}\left(\tilde{x}_{jk}, 120 \lambda r_{k}\right) \nonumber\\
&\leq&  C(n,d,C_{M}, C_{P})  \,r_{k} \, \epsilon_{3}^{\frac{1}{2n}}. \ees

Combining (\ref{6}) and (\ref{7''}), we get that 
\be  \label{10''''}  d(z ,P_{jk}) \leq C_{5} \,r_{k} \,  \epsilon_{3}^{\frac{1}{2n}} \quad \txt{for}\,\, z \in M \cap B_{120r_{k}}(\tilde{x}_{jk}), \ee
where $C_{5}= C_{5}(n,d,C_{M},C_{P})$.
Thus, for $ C_{5} \,  \epsilon_{3}^{\frac{1}{2n}} \leq \epsilon,$
we get (\ref{rf4}) which is the desired inequality.

Now, let us show that $P_{jk}$ and $Q_{jk}$ are close together, that is 

\be \label{rf5} d_{x_{jk}, 5 r_{k}}(P_{jk}, Q_{jk}) \leq 3 \epsilon \, r_{k}.  \ee 
Since $P_{jk}$ and $Q_{jk}$ are $n$-planes, it is enough to show 
\be \label{rf6} \sup_{y \in Q_{jk} \cap B_{5r_{k}}(x_{jk})} d(y, P_{jk}) \leq 3 \epsilon \, r_{k}. \ee
Let $y \in Q_{jk} \cap B_{5r_{k}}(x_{jk})$. By (\ref{rf2}), we get that $d(y,M) \leq \epsilon_{0} r_{k}$, and thus, there exists $y' \in M$ such that $|y - y'| \leq 2 \, \epsilon_{0} \, r_{k}$. Recalling that $ x_{jk} \in M \cap B_{ \frac{r_{k}}{6}}(\tilde{x}_{jk})$ (see (\ref{61})), we get 
\ben |y' - \tilde{x}_{jk}| \leq |y' - y| + |y - x_{jk}| + |x_{jk} - \tilde{x}_{jk}| \leq 2 \epsilon_{3} \, r_{k} + 5 r_{k} + \frac{r_{k}}{6} \leq 120 r_{k},\een
that is $y' \in B_{120 r_{k}}(\tilde{x}_{jk})$. Hence, by (\ref{rf4}), we get that $d(y', P_{jk}) \leq \epsilon \,r_{k}$, and using the fact that $\epsilon_{3} \leq \epsilon$, we get
\ben d(y, P_{jk}) \leq |y - y'| + d(y', P_{jk}) \leq 3 \epsilon \, r_{k} ,\een
which finishes the proof of (\ref{rf6}) and in particular (\ref{rf5}). \\

Before starting the proof of (\ref{rf3}), let us recall a little bit how the map $g$ was defined. In the proof of Theorem  \ref{t2} (see paragraph above (\ref{sigmas})) David and Toro constructed the smooth maps $\sigma_{k} \,\, \txt{and} \,\, f_{k}$ where $f_{0} = Id$ and $f_{k} = \sigma_{k-1} \circ \ldots \sigma_{0} \,\, \txt{for} \,\, k \geq 1$, and then defined the map $f = \displaystyle \lim_{k \to \infty} f_{k}$ defined on $\Sigma_{0}$, and finally the map $g$ was the extension of $f$ to the whole space. \\

In order to prove (\ref{rf3}), we will need the following inequality from \cite{DT1} (see proposition 5.1 page 19 in \cite{DT1})

\be \label{inrf} d(f_{k}(z), P_{jk}) \leq C(n,d) \, \epsilon \, r_{k}, \quad \forall \, z \in \Sigma_{0}, \, k\geq 0 \,\,\, \txt{and}\,\,\, j \in J_{k} , \,\,\, \txt{such that} \,\,\, f_{k}(z) \in B_{5r_{k}}(x_{jk}).\ee
We are finally ready to prove (\ref{rf3}). Let $w \in g(\Sigma_{0}) \cap B_{\frac{1}{10^{l_{0}+8}}}(0)$, and let $d_{0} := d(w, M)$. We would like to prove that $d_{0}=0$ (recall that $M$ is closed by assumption). Let $z \in \Sigma_{0}$ such that $ w = g(z)$. Notice that by (\ref{in}) (with $\epsilon$ instead of $\epsilon_{2}$), the definition of $f_{0}$, and the fact that $g$ and $f$ agree on $\Sigma_{0}$, we have 
\be \label{rf7} |w - z| = |g(z)-z| = |f(z) - f_{0}(z)| \leq C(n,d) \epsilon \, r_{0}. \ee

Recalling that $\Sigma_{0} = P_{i_{0}0}$, $\tilde{x}_{i_{0}0} = 0$ , $r_{0} = \frac{1}{10^{l_{0}+5}}$, and that $ x_{jk} \in B_{ \frac{r_{k}}{6}}(\tilde{x}_{jk})$ (see (\ref{61})), we get 

\bes \label{rf8} |z - x_{i_{0}0}| &\leq& |z - w| + | w - \tilde{x}_{i_{0}0}| + |\tilde{x}_{i_{0}0} - x_{i_{0}0}| \nonumber \\ 
&\leq& C(n,d) \epsilon \, r_{0}+ \frac{1}{10^{l_{0}+8}} + \frac{r_{0}}{6} \leq C_{6} \epsilon \, r_{0} + 2 r_{0} \leq 3 r_{0}, \ees
for $\epsilon$ such that $C_{6} \epsilon \leq 1$, where $C_{6} = C_{6}(n,d)$.
 Thus, $z \in P_{i_{0}0} \cap B_{5 r_{0}}(x_{i_{0}0})$, and by (\ref{rf5}), there is a point $z' \in Q_{i_{0}0}$ such that $ |z - z'| \leq 6 \epsilon \, r_{0}$.  Moreover, 
\be \label{rf10}  |z' - x_{i_{0}0}| \leq |z' - z| + |z - x_{i_{0}0}| \leq  6 \epsilon \, r_{0} + 3 r_{0} \leq 10 r_{0}, \ee
for $\epsilon < 1$. Thus, $z' \in Q_{i_{0}0} \cap B_{10 r_{0}}(x_{i_{0}0})$, and by (\ref{rf2}), we get that $d(z', M) \leq \epsilon_{3} \, r_{0}.$\\

Combining (\ref{rf7}), the line after (\ref{rf8}), the line before and the line after (\ref{rf10}), and the fact that $\epsilon_{3} \leq \epsilon$, we get

\be \label{rf11} d_{0} = d(w, M) \leq |w - z| + |z - z'| + d(z', M) \leq C_{6} \epsilon \, r_{0}+ 6 \epsilon r_{0} + \epsilon_{3} \, r_{0} = (C_{6} +7) \, \epsilon \, r_{0}   \leq \frac{r_{0}}{10},\ee
for $\epsilon$ such that $(C_{6} + 7) \, \epsilon \leq \frac{1}{10}$, where $C_{6} = C_{6}(n,d)$.

We proceed by contradiction. Suppose $d_{0} > 0$, then there exists $k \geq 0$ such that $r_{k+1} < d_{0} \leq r_{k}$. Notice that since $w = g(z)$, $z \in \Sigma_{0}$, and the maps $g$ and $f$ agree on $\Sigma_{0}$, then by (\ref{in}), we have 
\be \label{rf12} |w - f_{k}(z)| \leq C(n,d) \, \epsilon \, r_{k}. \ee

Now, by the definition of $d_{0}$, there exists $\xi \in M$ such that $ |\xi - w| \leq \frac{3}{2} d_{0}$. 
Using (\ref{rf11}) and the fact that $r_{0} = \frac{1}{10^{l_{0}+5}}$, we get 
\be \label{rf18} |\xi| \leq |\xi - w| + |w| \leq \frac{3}{2} \frac{r_{0}}{10} + \frac{1}{10^{l_{0}+8}}  \leq \frac{1}{10^{l_{0}+4}}, \ee 
 and thus by (\ref{63'}), there exists $j \in J_{k}$ such that $ \xi \in B_{\frac{3}{2}r_{k}}(x_{jk})$. \\
 
 Since both $k$ and $j$ are now fixed, consider the $n$-plane $P_{jk}$ and the point $x_{jk}$.
 By the line under (\ref{rf18}), the line under (\ref{rf12}), (\ref{rf12}), and the fact that $d_{0} \leq r_{k}$, we have 
 \be \label{rf100} |x_{jk} - f_{k}(z)| \leq |x_{jk} - \xi| + |\xi - w| + |w - f_{k}(z)| \leq \frac{3}{2}r_{k} + \frac{3}{2} d_{0} + C(n,d) \, \epsilon \, r_{k} \leq 3 r_{k} + C_{7} \, \epsilon \, r_{k} \leq 4r_{k}, \ee
 for $\epsilon$ such that $C_{7} \epsilon \leq 1$, where $C_{7} = C_{7}(n,d)$.
 Thus, inequality (\ref{inrf}) tell us that $d(f_{k}(z), P_{jk}) \leq C(n,d) \, \epsilon \, r_{k}$. Let $y \in P_{jk}$ such that $| y - f_{k}(z) | \leq C(n,d) \, \epsilon \, r_{k}$. Then, by (\ref{rf12}), the line below it, the line below (\ref{rf18}), and recalling that $d_{0} \leq r_{k}$, we get 

\be \label{rf17} | y - x_{jk}| \leq |y - f_{k}(z)| + |f_{k}(z) - w| + |w - \xi | + |\xi - x_{jk}| \leq C_{8} \, \epsilon \, r_{k} + 3 r_{k} \leq 5 r_{k} \ee
for $\epsilon$ such that $C_{8} \, \epsilon \leq 1$, where $C_{8} = C_{8}(n,d)$. Thus, $y \in P_{jk} \cap B_{5 r_{k}}(x_{jk})$, and by (\ref{rf5}) there exists $y' \in Q_{jk}$ such that $| y - y'| \leq 3 \epsilon \, r_{k}$. But then, $| y' - x_{jk}| \leq |y - y'| + |y - x_{jk}| \leq 10 \, r_{k}$; thus $ y' \in Q_{jk} \cap B_{10 r_{k}}(x_{jk})$ and by (\ref{rf2}) we get that $d(y',M) \leq \epsilon_{3} \, r_{k}$. \\

Finally, using (\ref{rf12}), the two lines before (\ref{rf17}), and the three lines below it, we get
\be \label{rf19} d_{0} = d(w,M) \leq |w - f_{k}(z)| + |f_{k}(z) - y| + |y - y'| + d(y',M) \leq C(n,d) \, \epsilon \, r_{k}  = C_{9} \epsilon \, r_{k} \leq r_{k+1}    \ee
for $\epsilon$ such that $C_{9} \epsilon \leq \frac{1}{10}$, where $C_{9}= C_{9}(n,d)$ which contradicts the fact that $d > r_{k+1}$.
This finishes the proof of (\ref{rf3}).\\

 Fix $\epsilon < \epsilon_{2} < 1$ such that the lines after (\ref{rf8}), (\ref{rf11}), (\ref{rf100}), (\ref{rf17}), and (\ref{rf19}) hold, and then fix $\epsilon_{3} \leq \epsilon \leq \epsilon_{2}$ such that inequality (\ref{c}) is satisfied with $\epsilon_{3}$ instead of $\epsilon_{0}$, (\ref{c'''}) is satisfied with $\epsilon$ instead of $\epsilon_{2}$ and $\epsilon_{3}$ instead of $\epsilon_{0}$,  the hypothesis of  Claim 1 is satisfied with $\epsilon_{3}$ instead of $\epsilon_{0}$, and such that the line below (\ref{10''''}) is satisfied. Writing $\epsilon_{3} = c_{10} \, \epsilon$, where $c_{10} = c_{10}(n,d,C_{M},C_{P})$, and replacing in (\ref{bb1rf}), we get (\ref{bb}). The proof that $g$ is bi-Lipschitz is the same as from Theorem \ref{MTT'}. 
\end{proof}

\section {The Poincar\'e Inequality (\ref{eqp1'}) is equivalent to the $p$-Poincar\'e inequality} \label{PIQ} 

 Let $(M, d_{0}, \mu)$ to be the metric measure space where $M \subset B_{2}(0)$ is $n$-Ahlfors regular rectifiable set in $\mathbb{R}^{n+d}$, $\mu =$  \( \mathcal{H}^{n} \mres M\) is the Hausdorff measure restricted to $M$, and $d_{0}$ is the restriction of the standard Euclidean distance in $\mathbb{R}^{n+d}$ to $M$. In this section, we prove Theorem \ref{epi}, which states that in the setting described above, the Poincar\'e inequality (\ref{eqp1'}) is equivalent to the $p$-Poincar\'e  inequality (\ref{eqp2'}) and the $Lip$-Poincar\'e inequality (\ref{eqp3'}).\\

We prove that $(iii) \implies (ii) \implies (i) \implies (iii)$. In fact, $(iii) \implies (ii)$ is proved in \cite{M1}. The fact that $(ii) \implies (i)$ follows from a theorem in \cite{K1} where Keith proves the equivalence between $p$-Poincar\'e inequalities and $Lip$-Poincar\'e inequalities. Finally, to prove $(i) \implies (iii)$, we use the well known fact that $X$ supporting a $p$-Poincar\'e inequality is equivalent to having inequality (\ref{eqp2}) hold for all measurable functions $u$ on $X$ and all $p$-weak upper gradients $\rho$ of $u$. Then, we show that $|\nabla^{M}f|$ is a $p$-weak upper gradient of $f$, when $f$ is a Lipschitz function on $\mathbb{R}^{n+d}$. \\

Let us start with stating the theorems that we need, as mentioned in the paragraph above.

\begin{theorem} (see \cite{K1}, Theorem 2) \label{theoremqc}
Let $p \geq 1$, and let $(X,d,\nu)$ be a complete metric measure space, with $\nu$ a doubling measure. Then, the following are equivalent:
\begin{itemize}
\item $(X,d,\nu)$ admits a $p$-Poincar\'e inequality for all measurable functions $u$ on $X$.
\item $(X,d,\nu)$ admits a $Lip$-Poincar\'e inequality for all Lipschitz functions $f$ on $X$.
\end{itemize}
\end{theorem}

\begin{theorem} (see \cite{BB}, Proposition 4.13) \label{propbjorns}
Let $p \geq 1$, and let $(X,d,\nu)$ be a metric measure space. Then, the following are equivalent:
\begin{itemize}
\item Inequality (\ref{eqp2}) holds for all measurable (resp. Lipschitz) functions $u$ on $X$ and all upper gradients $\rho$ of $u$.
\item Inequality (\ref{eqp2}) holds for all measurable (resp. Lipschitz) functions $u$ on $X$ and all $p$-weak upper gradients $\rho$ of $u$.
\end{itemize}
\end{theorem}

Before stating the theorem we need from \cite{M1}, let us make a remark on how the metric balls looks like in the metric measure space $(M, d_{0}, \mu)$. In fact, fix $x \in M$ and $r>0$. It is easy to see that 
\be \label{ballslook} B^{M}_{r}(x) = B_{r}(x) \cap M, \ee
where $B_{r}(x)$ denotes the Euclidean ball in $\mathbb{R}^{n+d}$ of center $x$ and radius $r$. 

\begin{theorem} \label{lemmaqc} (see \cite{M1} Corollary 5.8) \footnote{ Notice that Corollary 5.8 in \cite{M1} is stated and proved in the ambient space $\mathbb{R}^{n+1}$. However, the proof of Corollary 5.8 in \cite{M1} is independent from the co-dimension $d$ of $M$. Thus the exact same statement holds here in the higher co-dimension case. Moreover, notice that in Corollary 5.8, the Poincar\'e inequality assumed is (\ref{eqp1'}) but for $p= \lambda = 2$. This results in getting the Poincar\'e inequality (\ref{eqp3'}) but also for $p=\lambda = 2$. However, it is easy to see that one can assume the Poincar\'e inequality (\ref{eqp1'}) for any $p \geq 1$ and $\lambda \geq 1$, and get inequality (\ref{eqp3'}) for the same $p$ and $\lambda$ that one started with.} Let $(M, d_{0}, \mu)$ be as above. Assume that $M$ satisfies $(iii)$. Then, $M$ satisfies $(ii)$.
\end{theorem}

To show that $|\nabla^{M}f|$ is a $p$-weak upper gradient of $f$, when $f$ is a Lipschitz function on $\mathbb{R}^{n+d}$, we need the following lemma from \cite{BB}:

\begin{lemma} (see \cite{BB}, Lemma 1.42) \label{lemmabjorns}
Let $p \geq 1$ and let $(M, d_{0}, \mu)$ be as above. Suppose that $E \subset M$, with $\mu(E) = 0$. Denote by $\Gamma(M)$ the set of all rectifiable curves in $M$, and let 
\ben \Gamma_{E} = \left \{ \gamma \in \Gamma(M), \,\, \txt{such that} \,\, \mathcal{L}^{*}_{1}(\gamma^{-1}(E)) \neq 0 \right \}, \een
where $\mathcal{L}^{*}_{1}$ denotes the Lebesgue outer measure on $\mathbb{R}$.
Then, $ \txt{Mod}_{p}(\Gamma_{E}) = 0$.
\end{lemma}

\begin{proposition} \label{lemmaqc2}
 Let $(M, d_{0}, \mu)$ be as above, and suppose $f$ be a Lipschitz function on $\mathbb{R}^{n+d}$. Then, $|\nabla^{M}f|$ (or more precisely, any non-negative extension of $|\nabla^{M}f|$ to the whole space $M$) is a $p$-weak upper gradient of $f|_{M}$, the restriction of $f$ on $M$. \end{proposition}
 
\begin{proof}

Since $f$ Lipschitz on $\mathbb{R}^{n+d}$, we know that $\nabla^{M}f$ exists $\mu$-almost everywhere. Let
\ben  E = \left \{ x \in M \,\,\, \txt{such that} \,\,\, \nabla^{M}f(x) \,\, \txt{does not exist}  \right \}.\een
Then, $\mu(E)=0$, and by Lemma \ref{lemmabjorns}, we know that Mod$_{p}(\Gamma_{E}) = 0$. Now, let $\gamma$ be a rectifiable curve in $M$, parametrized by arc length, such that $\gamma \notin \Gamma_{E}$. Then, $\mathcal{L}_{1}(\gamma^{-1}(E)) = 0$. Moreover, Since $f \circ \gamma$ is Lipschitz, and thus absolutely continuous on $[0, l_{\gamma}]$, we have

\bes \label{np1} \big | f|_{M}(\gamma(0)) - f|_{M}(\gamma(l_{\gamma})) \big |  &=&   | f(\gamma(0)) - f(\gamma(l_{\gamma}))| \nonumber \\ &=&  \left | \int_{0}^{l_{\gamma}} (f \circ \gamma)'(t) \, dt \right | \nonumber \\ &=&  \left | \int_{t \in [0, l_{\gamma}]; \, \gamma(t) \notin E} (f \circ \gamma)'(t) \, dt \right | \leq  \int_{t \in [0, l_{\gamma}]; \, \gamma(t) \notin E} |(f \circ \gamma)'(t)| \, dt \ees

Let $t \in [0, l_{\gamma}]$ such that $\gamma(t) \notin E$. Then, $T_{\gamma(t)}M$ exists, and $\nabla^{M}f(\gamma(t)) \in T_{\gamma(t)}M$. We first show that 
\be \label{np4} | (f \circ \gamma)'(t)| \leq | \nabla^{M} f (\gamma(t))|. \ee
Since $\gamma '(t) \in T_{\gamma(t)}M$ \footnote{This follows directly from the facts that for any sequence $r \to 0$, we have $\gamma'(t) = \displaystyle \lim_{r \to 0} \displaystyle \frac{\gamma(t+r) - \gamma(t)}{r}$ and  $\displaystyle \lim_{r \to 0} \displaystyle \sup_{y \in \frac{M - \gamma(t)}{r}} d(y, T_{\gamma(t)}M) = 0$.}is a unit vector, then by Rademacher's Theorem, we have

\be \label{np3} \lim_{h \to 0} \frac{| f \big (\gamma(t) +h \gamma '(t) \big ) - f \big(\gamma(t)\big) - h  <\nabla^{M} f (\gamma(t)), \gamma '(t) >|}{h} = 0. \ee

Now, for any $-t < h < l_{\gamma} - t$, we have
\be \begin{split} \label{np5} \frac{|f \big (\gamma(t+h) \big ) - f \big(\gamma(t)\big)|}{h} & \leq \\ &\frac{|f \big (\gamma(t+h) \big ) - f \big (\gamma(t) + h \gamma '(t) \big )|}{h} + \frac{|f \big (\gamma(t) + h \gamma '(t) \big ) - f\big(\gamma(t)\big)|}{h} \nonumber \\
&\leq L_{f}\, \frac{| \gamma(t+h) - \gamma(t) - h \gamma '(t) |}{h} + \frac{|f \big (\gamma(t) + h \gamma '(t) \big ) - f\big(\gamma(t)\big)|}{h} \, , \end{split} \ee
where in the last step, we used the fact that $f$ is Lipschitz on $\mathbb{R}^{n+d}$.\\ 

Taking the limit as $h \to 0$ on both sides of (\ref{np5}), and using (\ref{np3}) and the fact that $\gamma '(t)$ is a unit vector, we get
\ben | (f \circ \gamma)'(t)| \leq | \nabla^{M} f(\gamma(t)) \cdot \gamma ' (t)| \leq  | \nabla^{M} f(\gamma(t))| \een
which is exactly (\ref{np4}). Replacing (\ref{np4}) in (\ref{np1}), we get 

\be \label{np8} \big | f|_{M}(\gamma(0)) - f|_{M}(\gamma(l_{\gamma}))\big | \leq \int_{t \in [0, l_{\gamma}]; \, \gamma(t) \notin E} | \nabla^{M} f (\gamma(t))| \, dt. \ee

Now, define the map $G : M \to [0, \infty]$ to be any non-negative extension of $|\nabla^{M}f|$ to the whole space $M$ (that is, $G(x) = |\nabla^{M}f(x)|$ on $M \setminus E$, which means that $G = |\nabla^{M}f|\,  \mu$-a.e.).  Plugging back in (\ref{np8}), we get

\bes \label{np7} | f|_{M}(\gamma(0)) - f|_{M}(\gamma(l_{\gamma}))| &\leq& \int_{t \in [0, l_{\gamma}]; \, \gamma(t) \notin E} G(\gamma(t)) \, dt \nonumber \\ &=& \int_{t \in [0, l_{\gamma}]; \, \gamma(t) \notin E} G(\gamma(t)) \, dt + \int_{t \in [0, l_{\gamma}]; \, \gamma(t) \in E} G(\gamma(t)) \, dt \nonumber \\  &=& \int_{0}^{l_{\gamma}} G\big(\gamma (t)\big) dt = \int_{\gamma} G \, ds. \ees \footnote{ The function G defined here is clearly measurable. However, since any non-negative measurable function coincides $\mu$-almost everywhere with a non-negative Borel function (see \cite{BB}, Proposition 1.2), we can assume, without any loss of generality that $G$ is Borel. In this case, $\int_{\gamma}G \, ds$ is well defined for any rectifiable curve $\gamma$ in $M$, and we do not need to worry about the last step in (\ref{np7}).}
This finishes the proof that $G$ is a $p$-weak upper gradient of $f|_{M}$.
\end{proof}

We are finally ready to prove Theorem \ref{epi}: \\

\textbf{\underline{\textit{Proof of Theorem \ref{epi}:}}}

\begin{proof}
We prove $ (iii) \implies (ii) \implies (i) \implies (iii)$:\\

 \noindent $(iii) \implies (ii)$:\\ This is exactly Theorem \ref{lemmaqc}.\\
 
 \noindent $(ii) \implies (i)$: \\ Notice that by using (\ref{ballslook}), we will be done if we apply Theorem \ref{theoremqc} to the metric measure space $(M, \mu, d_{0})$. In fact, $M$ is complete since it is closed and bounded. Moreover, the fact that $\mu$ is doubling follows from (\ref{ballslook}) and the Ahlfors regularity of $\mu$. Hence, we can apply Theorem \ref{theoremqc} to $(M, \mu, d_{0})$.\\

\noindent $(i) \implies (iii)$: \\
Notice that by Theorem \ref{propbjorns}, we know that $(i)$ implies that inequality (\ref{eqp2}) holds for all measurable functions $u$ on $M$ and all $p$-weak upper gradients $\rho$ of $u$. Let $f$ be a Lipschitz function $f$ on $\mathbb{R}^{n+d}$, and fix $x \in M$ and $r>0$. Then, $f|_{M}$ is a Lipschitz function on $M$, and by Lemma \ref{lemmaqc2}, $|\nabla^{M}f|$ agrees $\mu$-almost everywhere with $G$, a $p$-weak upper gradient of $f|_{M}$. Applying (\ref{eqp2}) for $u = f|_{M}$, $\rho = G$, and the ball $B = B_{r}(x) \cap M$, we get 
\ben  \dashint_{B_{r}(x)} \left| f(y) - f_{x,r}\right| d \mu(y) \leq \kappa r \left(\,\dashint_{B_{\lambda r}(x)} G(y)^{2} \, d \mu(y) \right)^{\frac{1}{2}} = \kappa r \left(\,\dashint_{B_{\lambda r}(x)} (|\nabla^{M}f|(y))^{2} \, d \mu(y) \right)^{\frac{1}{2}}  \een
hence finishing the proof
\end{proof}

\section{The conclusion of Theorem \ref{MTT'} is optimal} \label{MCSHH}

In this section, we prove Theorem \ref{construct} by giving an example of a non-Reifenberg flat, $2$-Ahlfors regular rectifiable set $M \subset \mathbb{R}^{3}$ that satisfies the Carleson condition (\ref{103}) and the Poincar\'e-type inequality (\ref{eqp}).  \\

To construct this example, we use the well known fact that Lipschitz domains support a $p$-Poincar\'e-type inequality, together with Theorem \ref{epi} that allows us to go from a $p$-Poincar\'e inequality to the Poincar\'e inequality (\ref{eqp1'}). \\

In order to keep track of where the balls live, $B^{2}_{r}(x)$ will denote the Euclidean ball in $\mathbb{R}^{2}$ of center $x$ and radius $r$, whereas $B^{3}_{r}(x)$ will be that in $\mathbb{R}^{3}$. Moreover diam($A$) denotes the diameter of a set $A$.

\begin{definition}
We say that a bounded set $A \subset \mathbb{R}^{2}$ satisfies the \emph{corkscrew condition} if there exists $\delta >0$ such that for all $x \in \bar{A}$ and $0<r \leq \txt{diam}(A)$, the set $B^{2}_{r}(x) \cap A$ contains a ball with radius $\delta r$.
\end{definition} 

\begin{definition}
We say that an open, bounded set $A \subset \mathbb{R}^{2}$ is Lipschitz domain if the boundary of $A$, $\partial A$ can be written, locally, as a graph of a Lipschitz function. More precisely, A is a Lipschitz domain if for every point $x \in \partial A$ there exists a radius $r>0$ and a bijective map $h_{x}: B^{2}_{r}(x) \to B^{2}_{1}(0)$ such that the following holds:

\begin{itemize}
\item $h_{x}$ and $h_{x}^{-1}$ are Lipschitz continuous.
\item $h_{x}(\partial A \cap B^{2}_{r}(x)) = Q_{0}$,
and
\item $h_{x}(A \cap B^{2}_{r}(x)) = Q_{1}$,
\end{itemize}
where $Q_{0} = \left\{ (x_{1}, x_{2}) \in B^{2}_{1}(0); x_{2} = 0 \right\}$ and $Q_{1} = \left\{ (x_{1}, x_{2}) \in B^{2}_{1}(0); x_{2} > 0 \right\}$.
\end{definition}

In \cite{BS}, J. Bj\"{o}rn and N. Shanmugalingam prove that Lipschitz domains support $p$-Poincar\'e-type inequalities: 

\begin{theorem} \label{Ht1} (see \cite{BS} Theorem 4.4) 
Consider the Hausdorff measure $\mathcal{H}^{2}$ on $\mathbb{R}^{2}$. Let $\Omega$ be any Lipschitz domain on $\mathbb{R}^{2}$. Then, $\Omega$ supports a 2-Poincar\'e-type inequality, that is there exist constants $\kappa\geq1$ and $\lambda\geq1$ such that for every $x \in \bar{ \Omega}$, and $r>0$, and for every Lipschitz function $u: \Omega \rightarrow \mathbb{R}$ and any upper gradient $\rho$ of $u$ in $\Omega$, the following holds
\be \label{Heq1} \dashint_{B^{2}_{r}(x) \cap \Omega} \left| u(y) - u_{x,r} \right| \, d \mathcal{H}^{2}(y) \leq \kappa \, r \, \left(\,\dashint_{B^{2}_{\lambda r}(x) \cap \Omega} \rho(y)^{2} \, d \mathcal{H}^{2}(y) \right)^{\frac{1}{2}}, \ee
where  $u_{x,r} := \dashint_{B^{2}_{r}(x) \cap \Omega} u \, d \mathcal{H}^{2}$.
\end{theorem}

We are now ready to construct our example. Let $\Omega := B^{2}_{1}(0) \setminus Q$ where $Q$ is the closed square of center $(\frac{1}{2},0)$, and side $l = \frac{1}{10}$. Since $\Omega$ is a Lipschitz domain, by Theorem \ref{Ht1}, it supports the 2-Poincar\'e-type inequality (\ref{Heq1}). \\

\textbf{\underline{\textit{Proof of Theorem \ref{construct}:}}}

\begin{proof} Let $\Omega$ be as in the construction above, and let $M: = \bar{\Omega} \times \{0\} \subset \mathbb{R}^{3}$. We prove this theorem for $n=2$, $d = 1$, and  $\mu =$  \( \mathcal{H}^{2} \mres M\). However, with a similar construction\footnote{ In general, we take $\Omega := B^{n}_{1}(0) \setminus Q$ where $Q$ is the closed $n$-cube of center $(\frac{1}{2}, \underbrace{0, \ldots,0}_\text{$n-1$-times})$, and side $l = \frac{1}{10}$. Then, $M := \bar{\Omega} \times \underbrace{(0, \ldots, 0)}_\text{$d$-times}$.}, the theorem holds for any $n \geq 2$ and $d \geq 1$.\\

It  is trivial to see that $M$ is a rectifiable non-Reifenberg flat set. To see that $M$ is 2-Ahlfors regular, first note that $M$ is closed by construction. So, we show that there exists a constant $C_{M} \geq 1$ such that for every $x \in M$ and $0<r \leq 1$, we have

\be \label{Heq2} C_{M}^{-1} \, r^{2}  \leq  \mu(M\cap B^{3}_{r}(x)) \leq C_{M} \, r^{2}. \ee
By the definition of $\mu$ and the construction of $M$, proving (\ref{Heq2}) translates to proving that for every $\bar{x} \in \bar{\Omega}$ and $0<r \leq 1$, 

\be \label{Heq3} C_{M}^{-1} \, r^{2}  \leq  \mathcal{H}^{2} (\bar{\Omega} \cap B^{2}_{r}(\bar{x})) \leq C_{M} \, r^{2}. \ee

The right hand side of (\ref{Heq3}) is trivial since $\mathcal{H}^{2}(\bar{\Omega} \cap B^{2}_{r}(\bar{x})) \leq \mathcal{H}^{2}(B^{2}_{r}(\bar{x})) = \omega_{2} \, r^{2}$. For the left hand side, notice that since $\Omega$ is a Lipschitz domain, then it is automatically a corkscrew domain, and thus there exists an $\delta >0$, such that for every $\bar{x} \in \bar{\Omega}$ and for every $0<r \leq \txt{diam}(\Omega) = 1$, there is a ball $B^{2}_{\delta r}(\bar{x}) \subset \bar{\Omega} \cap B^{2}_{r}(\bar{x})$. So, $\omega_{2} \, \delta^{2}r^{2} =\mathcal{H}^{2}( B^{2}_{\delta r}(\bar{x})) \leq \mathcal{H}^{2}(\bar{\Omega} \cap B^{2}_{r}(\bar{x})) $, and the proof of (\ref{Heq3}) is done.\\

Let us now prove that the Carleson-type condition (\ref{103}) holds. Let $\epsilon_{0}$ be the constant from the statement of Theorem \ref{MTT'}. Since $M$ has co-dimension 1, (\ref{103}) can be written as (\ref{103old}), and thus proving (\ref{103}) translates to proving 
\be \label{Heq4} \sup_{x \in A \cap B^{3}_{1}(0)} \,\, \int_{0}^{1} \left(\,\dashint_{B^{3}_{r}(x)} |\nu(y) - \nu_{x,r}|^{2} \, d \mu \right) \frac{dr}{r} < \epsilon_{0}^{2} ,\ee
where $\nu$ denotes the unit normal to $M$ and $\nu_{x,r} := \dashint_{B^{3}_{r}(x)} \nu \, d \mu$.
But for $\mu$-almost every $y$, $\nu(y)$ exists and $\nu(y) = <0,0,1>$. Thus, the left hand side of (\ref{Heq4}) is always 0, and (\ref{Heq4}) is satisfied. \\

Finally, let us prove that $M$ satisfies the following Poincar\'e inequality

\be \label{Heq6ag} \dashint_{B^{3}_{r}(x)} \left| f(y) - f_{x,r} \right| \, d \mu(y) \leq \kappa \, r \, \left(\,\dashint_{B^{3}_{\lambda r}(x)}|\nabla^{M}f(y)|^{2} \, d \mu(y) \right)^{\frac{1}{2}},\ee
for some $\kappa \geq 1$ and $\lambda \geq 1$, and where $x \in M$, $r >0$, $f$ is a Lipschitz function on $\mathbb{R}^{3}$, and $f_{x,r} := \dashint_{B^{3}_{r}(x)} f \, d \mu$. By Theorem \ref{epi}, it suffices to show that 
 
 \be \label{Heq6agag} \dashint_{B^{3}_{r}(x)} \left| f(y) - f_{x,r} \right| \, d \mu(y) \leq \kappa \, r \, \left(\,\dashint_{B^{3}_{\lambda r}(x)} \rho(y)^{2} \, d \mu(y) \right)^{\frac{1}{2}},\ee
for some $\kappa \geq 1$ and $\lambda \geq 1$, and where $x \in M$, $r >0$, $f$ is a Lipschitz \footnote{Notice that $(i)$ in Theorem \ref{epi} states that inequality (\ref{eqp2'}) should hold for all measurable functions $f$ and not only Lipschitz functions. However, from the proof of Theorem \ref{epi}, we know that the theorem still holds if we restrict $(i)$ to Lipschitz functions only.} function on $M$, $\rho$ is an upper gradient of $f$ in $M$, and $f_{x,r} := \dashint_{B^{3}_{r}(x)} f \, d \mu$.\\

Let $f$ be a Lipschitz function on $M$, and $\rho$ an upper gradient of $f$ on $M$. Fix $x \in M$ and $r>0$. Let $\tilde{x} \in \bar{\Omega}$ such that $(\tilde{x},0) = x$, and define the functions $\tilde{f} : \Omega \to \mathbb{R}$ and $\tilde{\rho} : \Omega \to [0, \infty]$ such that $\tilde{f}(a,b) = f(a,b,0)$ and $\tilde{\rho}(a,b) = \rho(a,b,0)$. It is easy to see that $\tilde{f}$ is a Lipschitz function on $\Omega$, and $\tilde{\rho}$ is an upper gradient to $\tilde{f}$ in $\Omega$. Thus, by the definition of $\mu$, the construction of $M$, the fact that $\mathcal{H}^{2}(\bar{\Omega} \setminus \Omega) = 0$, and using (\ref{Heq1}) (for $x = \tilde{x}$, $u = \tilde{f}$, and $\rho = \tilde{\rho}$), we get 

\besn \dashint_{B^{3}_{r}(x)} \left| f(y) - f_{x,r} \right| \, d \mu(y) &=& \dashint_{B^{2}_{r}(\tilde{x}) \cap \Omega} \left| \tilde{f}(y) - \tilde{f}_{\tilde{x},r} \right| \, d \mathcal{H}^{2}(y) \\ 
&\leq& \kappa \, r \, \left(\,\dashint_{B^{2}_{\lambda r}(\tilde{x})\cap \Omega} \tilde{\rho}(y)^{2} \, d \mathcal{H}^{2}(y) \right)^{\frac{1}{2}} \\ 
 &=& \kappa \, r \, \left(\,\dashint_{B^{3}_{\lambda r}(x)} \rho(y)^{2} \, d \mu(y) \right)^{\frac{1}{2}}, \eesn
which is exactly (\ref{Heq6agag}) hence finishing the proof of this theorem
 \end{proof}
 
 \begin{remark} Notice that one could take away more that one square $Q$ from the ball $B^{2}_{1}(0)$ and still get the same result of this section. The important thing about the construction above is that $\Omega$ is a Lipschitz domain; Thus if we want to construct a set with $m$ holes that satisfies the hypotheses of Theorem \ref{MTT'}, all we need to do is make sure that the squares we take away from the ball $B^{2}_{1}(0)$ and are far away from each other (that is, they do not accumulate). That way, $\Omega \setminus \displaystyle \bigcup_{i=1}^{m} Q_{i}$ remains a Lipschitz domain and the rest of the argument follows directly.
\end{remark}

As mentioned in the introduction, the example constructed in Theorem \ref{construct} proves that the conclusion of Theorem \ref{MTT'} is optimal.

\begin{ack}
The author would like to thank T. Toro for her supervision, direction, and numerous insights into the subject of this project.

\end{ack}

\bibliography{newpaper}{}
\bibliographystyle{amsalpha}
\end{document}